\documentclass[9pt,fullpaper,twoside,web]{ieeecolor}
\usepackage{amsmath,amssymb,amsfonts}
\usepackage{graphicx}
\usepackage{color}
\usepackage{subfigure}
\usepackage{balance}
\usepackage{textcomp}
\usepackage{xcolor}
\usepackage{generic}
\usepackage{cite}
\usepackage{multicol,lipsum}
\usepackage[utf8x]{inputenc}
\usepackage{epsfig}
\usepackage{algorithm}
\usepackage{algorithmic}
\usepackage{multirow}
\usepackage{hyperref}
\usepackage{bm}
\usepackage{listings}
\usepackage{mathtools}
\newtheorem{remark}{Remark}
\newtheorem{definition}{Definition}
\newtheorem{theorem}{Theorem}
\newtheorem{proposition}{Proposition}

\newtheorem{example}{Example}

\newcommand{\oomit}[1]{}

\def\BibTeX{{\rm B\kern-.05em{\sc i\kern-.025em b}\kern-.08em
    T\kern-.1667em\lower.7ex\hbox{E}\kern-.125emX}}
\begin{document}
\title{Reach-avoid Verification Based on Convex Optimization}
\author{Bai Xue$^{1,2}$, Naijun Zhan$^{1,2}$, Martin Fr\"anzle$^3$, Ji Wang$^4$ and Wanwei Liu$^4$
\thanks{1. State Key Lab. of Computer Science, Institute of Software Chinese Academy of Sciences, Beijing, China (\{xuebai,znj\}@ios.ac.cn) }
\thanks{2. University of Chinese Academy of Sciences, Beijing, China}
\thanks{3. Carl von Ossietzky Universit\"at Oldenburg, Oldenburg, Germany (martin.fraenzle@uni-oldenburg.de)}
\thanks{4. National University of Defense Technology, Changsha, China  (\{wj,wwliu\}@nudt.edu.cn; jiwang@ios.ac.cn)}
}

\maketitle

\begin{abstract}
In this paper we propose novel optimization-based methods for verifying reach-avoid (or, eventuality) properties of continuous-time systems modelled by ordinary differential equations. Given a system, an initial set, a safe set and a target set of states, we say that the reach-avoid property holds if for all initial conditions in the initial set, any trajectory of the system starting at them will eventually, i.e.\ in unbounded yet finite time, enter the target set while remaining inside the safe set until that first target hit. Based on a discount value function, two sets of quantified constraints are derived for verifying the reach-avoid property via the computation of exponential/asymptotic guidance-barrier functions (they form a barrier escorting the system to the target set safely at an exponential or asymptotic rate). It is interesting to find that one set of constraints whose solution is termed exponential guidance-barrier functions is just a simplified version of the existing one derived from the moment based method, while the other one whose solution is termed asymptotic guidance-barrier functions is completely new. Furthermore, built upon this new set of constraints, we derive a set of more expressive constraints, which includes the aforementioned two sets of constraints as special instances, providing more chances for verifying the reach-avoid properties successfully. When the involved datum are polynomials, i.e., the initial set, safe set and target set are semi-algebraic, and the system has polynomial dynamics, the problem of solving these sets of constraints can be framed as a semi-definite optimization problem using sum-of-squares decomposition techniques and thus can be efficiently solved in polynomial time via interior point methods. Finally, several examples demonstrate the theoretical developments and performance of proposed methods.
\end{abstract}


\section{Introduction}
\label{Int}
Cyber-physical technology is integrated into an ever-growing range of physical devices and increasingly pervades our daily life \cite{lee2008cyber}. Examples of such systems range from  intelligent highway systems, to air traffic management systems, to computer and communication networks, to smart houses and smart supplies, etc. \cite{williams2003,braunl2008}. Many of the above-mentioned applications are safety-critical and require a rigorous guarantee of safe operation.

Among the many possible rigorous guarantees, verifying whether the system's dynamics (generally modelled by ordinary differential equations)  satisfy required properties is definitely in demand. Besides the more traditional properties such as stability and input-output performance \cite{khalil2002nonlinear}, properties of interest also includes safety and reachability. Safety verification answers the problem whether a system starting from any initial condition in some prescribed set cannot evolve to some unsafe region in the state space. On the other hand, reachability verification aims to show that for some or all initial conditions in some  prescribed set, the system will evolve to some target region in the state space \cite{prajna2007convex}. In existing literature, there are various extensions of safety and reachability such as reach-avoid properties to be verified. Reach-avoid verification is to verify weather all trajectories of the system starting from a specified initial set will enter a target set within a bounded time horizon or eventually while staying inside a safe set before the target hitting time. Since the reach-avoid property combines guarantees of safety by avoiding a set of unsafe states with the reachability property of reaching a target set and thus can formalize many important enginnering problems such as collision avoidance \cite{margellos2012} as well as target surveillance \cite{gillula2012}, its verification has turned out to be of fundamental importance in engineering. Mechanically providing mathematical guarantees for such scenarios is challenging but interesting due to practical applications mentioned above. 


One of the popular methods for the reach-avoid verification is computational reachability analysis, which involves the explicit computation of reachable states. It has been widely studied over the last three decades in several disciplines including control theory, computer science and applied mathematics \cite{franzle2019memory,althoff2021set}. In general, the exact computation of reach sets is impossible for dynamical and hybrid  systems \cite{Henzinger98,Zhan18}. Over-approximate reachability analysis is therefore studied in existing literature  for verification purposes (e.g., \cite{asarin2000}). Over-approximate reachabilty analysis computes an over-approximation (i.e., super-set) of the reach set based on set propagation techniques, which is an extension of traditional numerical methods for solving ordinary differential equations using set arithmetic rather than point arithmetic. If the computed over-approxiamtion does not intersect with unsafe states and is included in the target set, the reach-avoid property holds. Overly pessimistic over-approximations, however, render many properties unverifiable in practice, especially for large initial sets and/or large time horizons. This pessimism mainly arises due to the wrapping effect \cite{kulisch2001perspectives}, which is the propagation and accumulation of over-approximation errors through the iterative computation of reach sets. There are many techniques developed in existing literature for controlling the wrapping effect. One way is to use complex sets such as Taylor models \cite{berz1998verified,chen2012taylor} and polynomial zonotopes \cite{althoff2013reachability}  to over-approximate the reach set. On the other hand, as the extent of the wrapping effect correlates strongly with the size of the initial set, another way is to exploit subsets of the initial set for performing over-approximate reachability analysis via exploiting the (topological) structure of the system. For instance,  appropriate corner points of reach sets, called bracketing systems, are used in \cite{ramdani2009hybrid,eggers2015improving} to bound the complete reach sets when the systems under consideration are monotonic; \cite{xue2016reach} proposed the set-boundary reachability method for for continuous-time systems featuring a locally Lipschitz-continuous vector field, which just need to perform over-approximate reachability analysis of the initial set's boundary for the reach-avoid verification.

Another popular method is the optimization based method, which transforms the verification problem into a problem of determining the existence of solutions to a set of quantified constraints. This method avoids the explicit computation of reach sets and thus can handle verification with unbounded time horizons. A well-known method is the barrier certificate method, which was originally proposed in \cite{prajna2004safety,prajna2007framework} for safety verification of continuous and hybrid systems. The barrier certificate method was inspired by Lyapunov functions in control theory and relies on the computation of barrier certificates, which are a function of state satisfying a set of quantified inequalities on both the function itself and its Lie derivative along the flow of the system. In the state space, the zero level set of a barrier certificate separates an unsafe region from all system trajectories starting from a set of legally initial states and thus the existence of such a function provides an exact certificate/proof of system safety. Afterwards, a number of different kinds of barrier certificates were developed such as exponential barrier certificates and vector barrier certificates in the literature \cite{kong2013exponential,sogokon2018vector,dai2017barrier,bak2018t}, which mainly differ in their expressiveness. This method was also extended to reachability verification of continuous and hybrid systems. For instance, it was extended to the reach-avoid verification in  \cite{prajna2007convex}.

In this paper we study the reach-avoid verification problem of continuous-time systems modelled by ordinary differential equations in the framework of the optimization based method. The reach-avoid verification problem of interest is that given an initial set, a safe set and a target set, we verify whether any trajectory starting from the initial set will eventually enter the target set while remaining inside the safe set until the first target hit. The reach-avoid verification problem in our method is transformed into a problem of searching for so called guidance-barrier functions. The reason that we term guidance-barrier function is that the boundary of its certain sub-level set forms a barrier, escorting the system to the target set safely. Based on a discount value function, whose certain (sub) level set equals the set of all initial states enabling the satisfaction of reach-avoid properties, with the discount factor being larger than and equal to zero we first respectively derive two sets of quantified constraints whose solutions are termed exponential and asymptotic guidance barrier functions. If a solution to any of these two sets of constraints is found, the reach-avoid property is guaranteed. Based on the set of constraints associated with asymptotic guidance-barrier functions, we further construct a set of  more expressive constraints, which admits more solutions and formulates the aforementioned two sets of constraints as its special instances, and thus offers more possibilities of verifying the reach-avoid property successfully. When the datum are polynomials, i.e., the system has polynomial dynamics, the initial set, target set and safe set are semi-algebraic sets, the problem of solving these constraints can be reduced to a semi-definite programming problem which can be solved efficiently in polynomial time via interior point methods. Finally, several examples are provided to demonstrate theoretical developments of the proposed methods, and the experiment results show that they outperform existing ones.

The main contributions of this work are summarized below.
\begin{enumerate}
    \item A novel unified framework is proposed for the reach-avoid verification of continuous-time systems modelled by ordinary differential equations such that the verification problem is transformed into a problem of searching for exponential/asymptotic guidance-barrier functions. In this framework, two sets of quantified inequalities are derived based on a discount value function. One is a simplified version of the existing one from the moment based method in \cite{korda2014controller} however, the other one is completely new. 
    \item The differences and respective benefits of the aforementioned two sets of constraints are discussed in detail. Based on these discussions, a set of more expressive constraints is further developed such that the aforementioned two sets of constraints are its special cases.
\end{enumerate}

\subsection*{Related Work}
Among the many possible extensions beyond reachability analysis, reach-avoid analysis has turned out to be of fundamental importance in engineering, as it can formalize many important engineering problems such as collision avoidance \cite{margellos2012} and target surveillance \cite{gillula2012}. The reach-avoid problem comes in the two variants of computing a reach-avoid set and of verifying reach-avoid properties for systems featuring a given initial state set \cite{xue2016reach}. A reach-avoid set, also known as the capture basin in viability theory \cite{aubin2001viability}, is the maximal set of initial states such that a system starting from them is guaranteed to (eventually or within a given time horizon) reach a given target set while avoiding a given unsafe set till the target hit. The computation of reach-avoid sets is closely related to the reach-avoid verification, especially the computation of their under-approximations (i.e., subsets). If the given initial set is included in the computed under-approximation, it is evident that the reach-avoid property holds.  

\cite{prajna2007convex} presented a set of quantified constraints for the reach-avoid verification of continuous-time systems. One of constraints is that the Lie derivative of the computed function along the flow of the system is strictly negative over a set which is the closure of safe set minus target set. This condition is strict and thus limits its application. For instance, it is impossible for this constraint to be satisfied when the system has equilibrium states in this set. The derived constraints in the present work remedy this shortcoming.

On the other hand, the problem of computing reach-avoid sets of continuous-time systems has been investigated in the Hamilton-Jacobi reachability framework, e.g., \cite{mitchell2000level,bokanowski2010,margellos2011,fisac2015,fisac2015pursuit}, which links reach-avoid sets with viscosity solutions to Hamilton-Jacobi equations and finally reduces the problem of computing reach-avoid sets to the problem of addressing Hamilton-Jacobi equations. However, traditional numerical methods for solving Hamilton-Jacobi equations require gridding the state space such that the computed reach-avoid set is just an approximation, neither an over-approximation nor under-approximation, which cannot be used to reason formally on the system. More recently, via relaxing Hamilton-Jacobi equations, a method exploiting semi-definite programming for under-approximating reach-avoid sets has been suggested in \cite{xue2019}. Besides, moment-based programming methods were proposed for over- as well as under-approximating reach-avoid sets in \cite{henrion2013convex,korda2013inner,shia2014convex,majumdar2014,zhao2017control}. All aforementioned works are confined in that they compute reach-avoid sets over specified bounded time horizons and thus cannot be used to study the reach-avoid verification problem in this work. Studies on computing under-approximations of reach-avoid sets over unbounded time horizons are rare by the best of our knowledge, with only moment-based programming methods attempting to address them \cite{korda2014controller}. The constraints developed in the present work include the ones in \cite{korda2014controller} as a special case and thus have more capabilities for the reach-avoid verification.

A set of quantified constraints, derived by reduction from a pertinent set of equations, was recently derived in \cite{xue2020inner}. It addresses the computation of under-approximations of reach-avoid sets for \emph{discrete-time systems}. The present work is an extension of the work \cite{xue2020inner} via following the line of constructing quantified constraints. However, the construction is completely different, since the value function in \cite{xue2020inner} relies on a long-run average cost, which is constructed based on the polynomial function defining the target set. Such a characterization is not applicable to continuous-time systems, due to different topological dynamics between continuous- and discrete-time systems. For example, if trajectories for discrete-time systems starting from states outside a target set enter the target set, they may not touch the boundary of the target set. However, this assertion does not hold for continuous-time systems. Consequently, we define a value function of a new form, which is built upon the indicator function of the target set and a discount factor, and this new value function facilitates the gain of several sets of constraints for the reach-avoid verification of continuous-time systems.

This paper is structured as follows. In Section \ref{pre}, basic notions used throughout this paper and the reach-avoid verification problem of interest are introduced. Then our methods for the reach-avoid verification are elucidated in Section \ref{inner_app}. After demonstrating the performance of proposed methods on a series of examples in Section \ref{ex}, we conclude this paper in Section \ref{con}.

\section{Preliminaries}
\label{pre}
In this section we formally present the concepts of continuous-time systems and reach-avoid verification problem of interest in this paper. Before formulating them, let us introduce some basic notions used throughout this paper: for a function $v(\bm{x})$, $\bigtriangledown_{\bm{x}}v(\bm{x})$ denotes its gradient with respect to $\bm{x}$; $\mathbb{R}_{\geq 0}(\mathbb{R}_{>0})$ stands for the set of nonnegative (positive) real values in $\mathbb{R}$ with $\mathbb{R}$ being the set of real numbers; the closure of a set $\mathcal{X}$ is denoted by $\overline{\mathcal{X}}$, the complement by $\mathcal{X}^c$ and the boundary by $\partial \mathcal{X}$; the ring of all multivariate polynomials in a variable $\bm{x}$ is denoted by $\mathbb{R}[\bm{x}]$; vectors are denoted by boldface letters, and the transpose of a vector $\bm{x}$ is denoted by $\bm{x}^{\top}$; $\sum[\bm{x}]$ is used to represent the set of sum-of-squares polynomials over variables $\bm{x}$, i.e., \[\sum[\bm{x}]=\{p\in \mathbb{R}[\bm{x}]\mid p=\sum_{i=1}^{k} q_i^2,q_i\in \mathbb{R}[\bm{x}],i=1,\ldots,k\}.\]
All of semi-definite programs in this paper are formulated using Matlab package YALMIP \cite{lofberg2004} and solved by employing the academic version of the semi-definite programming solver MOSEK \cite{mosek2015mosek}, and the computations were performed on an i7-P51s 2.6GHz CPU with 32GB RAM running Windows 10 and MATLAB2020a.

\subsection{Preliminaries}

The continuous-time system of interest (or, \textbf{CTolS}) is a system whose dynamics are described by an ODE of the following form:
\begin{equation}
\label{sys}
\dot{\bm{x}}=\bm{f}(\bm{x}),
\bm{x}(0)=\bm{x}_0\in \mathbb{R}^n,
\end{equation}
where $\dot{\bm{x}}=\frac{d \bm{x}(t)}{d t}$ and $\bm{f}(\bm{x})=(f_1(\bm{x}),\ldots,f_n(\bm{x}))^{\top}$ with $f_i(\bm{x})\in \mathbb{R}[\bm{x}]$.

We denote the trajectory of system \textbf{CTolS} that originates from $\bm{x}_0 \in \mathbb{R}^n$ and is defined over the maximal time interval $[0,T_{\bm{x_0}})$ by $\bm{\phi}_{\bm{x}_0}(\cdot):[0,T_{\bm{x_0}})\rightarrow \mathbb{R}^n$. Consequently,
\[\bm{\phi}_{\bm{x}_0}(t):=\bm{x}(t), \forall t\in [0,T_{\bm{x_0}}), \text{ and } \bm{\phi}_{\bm{x}_0}(0) = \bm{x}_0,\] where $T_{\bm{x}_0}$ is either a positive value (i.e., $T_{\bm{x_0}}\in \mathbb{R}_{>0}$) or $\infty$.

Given a bounded safe set $\mathcal{X}$, an initial set $\mathcal{X}_0$ and a target set $\mathcal{X}_r$, where 
\begin{equation*}
\label{set}
\begin{split}
&\mathcal{X}=\{\bm{x}\in \mathbb{R}^n\mid h(\bm{x})<0\} \text{~with~} \partial \mathcal{X}=\{\bm{x}\in \mathbb{R}^n\mid h(\bm{x})=0\},\\
&\mathcal{X}_0=\{\bm{x}\in \mathbb{R}^n\mid l(\bm{x})<0\}, \text{~and~}\\
&\mathcal{X}_r=\{\bm{x}\in \mathbb{R}^n\mid g(\bm{x})<0\}
\end{split}
\end{equation*}
with $l(\bm{x}), h(\bm{x}),g (\bm{x})\in \mathbb{R}[\bm{x}]$, and $\mathcal{X}_0\subseteq \mathcal{X}$ and $\mathcal{X}_r\subseteq \mathcal{X}$, the reach-avoid property of interest is defined as follows.
\begin{definition}[Reach-Avoid Property]\label{RNS}
Given system \textbf{CTolS} with the safe set $\mathcal{X}$, initial set $\mathcal{X}_0$ and target set $\mathcal{X}_r$, we say that the reach-avoid property holds if for all initial conditions $\bm{x}_0\in \mathcal{X}_0$, any trajectory $\bm{x}(t)$ of system \textbf{CTolS} starting at $\bm{\phi}_{\bm{x}_0}(0)=\bm{x}_0$ satisfies \[\bm{\phi}_{\bm{x}_0}(T)\in \mathcal{X}_r \bigwedge \forall t\in [0,T]. \bm{\phi}_{\bm{x}_0}(t)\in \mathcal{X}\] for some $T>0$.
\end{definition}

The problem of interest in this work is the reach-avoid verification, i.e., verifying that system \textbf{CTolS} satisfies the reach-avoid property in Definition \ref{RNS}. We attempt to solve this problem within the framework of optimization based methods. Generally, such methods are sound but incomplete.

It is  worth remarking here that the proposed methods in the subsequent sections can be extended straightforwardly to the case with the initial set $\mathcal{X}_0$ and target set $\mathcal{X}_r$ being a finite union of multiple sets, i.e., $\mathcal{X}=\cup_{i=1}^{i_1} \mathcal{X}_i$, $\mathcal{X}_0=\cup_{k=1}^{k_1} \mathcal{X}_{0,k}$,  and $\mathcal{X}_r=\cup_{j=1}^{j_1} \mathcal{X}_{r,j}$, where $\mathcal{X}_i=\{\bm{x}\in \mathbb{R}^n \mid h_i(\bm{x})<0\}$, $\mathcal{X}_{0,k}=\{\bm{x}\in \mathbb{R}^n \mid l_{k}(\bm{x})<0\}$ and $\mathcal{X}_{r,j}=\{\bm{x}\in \mathbb{R}^n \mid g_{j}(\bm{x})<0\}$, via using $h(\bm{x})=\max_{i\in \{1,\ldots,i_1\}} h_i(\bm{x})$, $l(\bm{x})=\max_{k\in \{1,\ldots,k_1\}} l_k(\bm{x})$  and $g(\bm{x})=\max_{j\in \{1,\ldots,j_1\}} g_j(\bm{x})$. For ease of exposition, we mainly take $i_1=k_1=j_1=1$ in the sequel, if not explicitly stated otherwise.

\subsection{Existing Methods}
For the convenience of comparisons, in this subsection we recall two sets of quantified constraints in existing literature for verifying the reach-avoid property in Definition \ref{RNS}. The first one is from \cite{prajna2007convex}, while the other one is from \cite{korda2014controller}.

\begin{proposition} \cite{prajna2007convex}
 Suppose that there exists a continuously differentiable function $v(\bm{x}): \overline{\mathcal{X}}\rightarrow \mathbb{R}$ satisfying 
\begin{align}
&v(\bm{x})\leq 0, \forall \bm{x}\in \mathcal{X}_0\label{041}  \\
&v(\bm{x})>0, \forall \bm{x}\in \overline{\partial \mathcal{X}\setminus \partial \mathcal{X}_r},\label{042} \\
&\bigtriangledown_{\bm{x}}v(\bm{x}) \cdot \bm{f}(\bm{x})<0, \forall \bm{x}\in \overline{\mathcal{X}\setminus \mathcal{X}_r},\label{043}
\end{align}
Then the reach-avoid property in Definition \ref{RNS} holds.
\end{proposition}
One of the drawbacks of constraints \eqref{041}-\eqref{043} in the reach-avoid verification is that it cannot deal with the case with an equilibrium being inside $\overline{\mathcal{X}\setminus \mathcal{X}_r}$, since  $\bm{f}(\bm{x}_0)=0$ implies $\bigtriangledown_{\bm{x}}v(\bm{x}) \cdot \bm{f}(\bm{x})\mid_{\bm{x}=\bm{x}_0}=0$.

Besides, if the reach-avoid property in Definition \ref{RNS} holds, the initial set $\mathcal{X}_0$ must be a subset of the reach-avoid set $\mathcal{RA}$, which is the set of all initial states guaranteeing the satisfaction of the reach-avoid property, i.e.
\begin{equation*}
\label{ras}
\mathcal{RA}=\left\{\bm{x}_0 \in \mathbb{R}^n \middle|\;
\begin{aligned}
&\exists t \in  \mathbb{R}_{\geq 0}. \bm{\phi}_{\bm{x}_0}(t)\in \mathcal{X}_r \\
&\bigwedge \forall \tau \in [0,t]. \bm{\phi}_{\bm{x}_0}(\tau) \in \mathcal{X}
\end{aligned}
\right\}.
\end{equation*}
Therefore, the method for computing under-approximations of the reach-avoid set $\mathcal{RA}$ can be used for the reach-avoid verification. By adding the condition \[v(\bm{x})< 0, \forall \bm{x}\in \mathcal{X}_0\] into constraint \cite[(18)]{korda2014controller}, which is originally developed for under-approximating the reach-avoid set $\mathcal{RA}$, we can obtain a set of quantified constraints as shown in Proposition \ref{korda2014} for the reach-avoid verification. 
\begin{proposition} 
\label{korda2014}
 Suppose that there exists a continuously differentiable function $v(\bm{x}): \overline{\mathcal{X}}\rightarrow \mathbb{R}$ and a continuous function $w(\bm{x}): \overline{\mathcal{X}}\rightarrow \mathbb{R}$  satisfying 
\begin{align}
&v(\bm{x})< 0, \forall \bm{x}\in \mathcal{X}_0\label{0411} \\
&\bigtriangledown_{\bm{x}}v(\bm{x}) \cdot \bm{f}(\bm{x})\leq \beta v(\bm{x}), \forall \bm{x}\in \overline{\mathcal{X}\setminus \mathcal{X}_r},\label{0421}\\
&w(\bm{x})\geq 0, \forall \bm{x}\in \overline{\mathcal{X}\setminus \mathcal{X}_r}, \label{0431}\\
&w(\bm{x})\geq v(\bm{x})+1, \forall \bm{x}\in \overline{\mathcal{X}\setminus \mathcal{X}_r}, \label{0441}\\
&v(\bm{x})\geq 0, \forall \bm{x}\in \partial \mathcal{X}, \label{0451}
\end{align}
where $\beta>0$ is a user-defined value, then the reach-avoid property in Definition \ref{RNS} holds.
\end{proposition}

\section{Reach-avoid Verification}
\label{inner_app}
This section presents our optimization based methods for the reach-avoid verification. Based on a discount value function, which is defined based on trajectories of a switched system and introduced in Subsection \ref{ISS}, two sets of quantified constraints are first respectively derived when the discount factor is respectively equal to zero and larger than zero. Once a solution (termed exponential or asymptotic guidance-barrier function) to  any of these two sets of constraints is found, the reach-avoid property in Definition \ref{RNS} is verified successfully. Furthermore, inspired by the set of constraints obtained when the discount factor is zero, a set of more expressive constraints is constructed for the reach-avoid verification.

\subsection{Induced Switched Systems}
\label{ISS}
This subsection introduces a switched system, which is built upon system \textbf{CTolS}. This switched system is constructed by requiring the state of system \textbf{CTolS} to stay still  when the complement of the safe set $\mathcal{X}$ is reached. For the sake of brevity, only trajectories of the induced switched system, also called \textbf{CSPS}, are introduced.

\begin{definition}
\label{tra}
Given system \textbf{CSPS} with an initial state $\bm{x}_0\in  \overline{\mathcal{X}}$, if there is a function $\bm{x}(\cdot): \mathbb{R}_{\geq 0}\rightarrow \mathbb{R}^n$ with $\bm{x}(0)=\bm{x}_0$ such that it satisfies the dynamics defined by $\dot{\bm{x}}=\widehat{\bm{f}}(\bm{x})$,
where 
\begin{equation}
\label{path_1}  
\widehat{\bm{f}}(\bm{x}): =1_{\mathcal{X}}(\bm{x})\cdot \bm{f}(\bm{x}),
  \end{equation}
 with $1_{\mathcal{X}}(\cdot): \overline{\mathcal{X}}\rightarrow \{0,1\}$ representing the indicator function of the set $\mathcal{X}$,  i.e.,
\[1_{\mathcal{X}}(\bm{x}):=\begin{cases}
   1, \quad \text{if }\bm{x}\in \mathcal{X},\\
   0, \quad \text{if }\bm{x}\notin \mathcal{X},
\end{cases}\]
then the trajectory $\widehat{\bm{\phi}}_{\bm{x}_0}(\cdot):\mathbb{R}_{\geq 0}\rightarrow \mathbb{R}^n$, induced by $\bm{x}_0$, of system \textbf{CSPS} is defined as follows:
\[\widehat{\bm{\phi}}_{\bm{x}_0}(t):=\bm{x}(t), \forall t\in \mathbb{R}_{\geq 0}.\]
\end{definition}

It is observed that the set $\overline{\mathcal{X}}$ is an invariant set for system \textbf{CSPS}. Also, if $\bm{x}_0\in \mathcal{X}$ and there exists $T\geq \mathbb{R}_{\geq 0}$ such that $\widehat{\bm{\phi}}_{\bm{x}_0}(t)\in \mathcal{X}$ for $t\in [0,T]$, we have that 
\[\widehat{\bm{\phi}}_{\bm{x}_0}(t)=\bm{\phi}_{\bm{x}_0}(t), \forall t\in [0,T].\]
Also, trajectories of system \textbf{CSPS} evolving in the viable set $\overline{\mathcal{X}}$ can be classified into three  disjoint groups: 
\begin{enumerate}
\item trajectories entering the set $\mathcal{X}_r$ in finite time. It is worth remarking here that these trajectories will not leave the safe set $\mathcal{X}$ before reaching the target set $\mathcal{X}_r$. Due to the fact that \[\widehat{\bm{\phi}}_{\bm{x}_0}(t)=\bm{\phi}_{\bm{x}_0}(t), \forall t\in [0,T]\]
where $\bm{x}_0\in \mathcal{X}$ and $T\in \mathbb{R}_{\geq 0}$ is a time instant such that $\widehat{\bm{\phi}}_{\bm{x}_0}(t)\in \mathcal{X}$ for $t\in [0,T]$, we conclude that the set of initial states deriving these trajectories equals the reach-avoid set $\mathcal{RA}$;
\item trajectories entering the set $\partial \mathcal{X}$ in finite time, but never entering the target set $\mathcal{X}_r$;
\item trajectories staying in the set $\mathcal{X}\setminus \mathcal{X}_r$ for all time.
\end{enumerate}

\subsection{Discount Value Functions}
\label{GBF2}
The discount value function aforementioned is introduced in this subsection.  With a non-negative discount factor $\beta$, the discount value function $V(\bm{x}):\overline{\mathcal{X}}\rightarrow \mathbb{R}$ with a non-negative discount factor $\beta$ is defined in the following form:
\begin{equation}
\label{value2}
V(\bm{x}):=\sup_{t\in \mathbb{R}_{\geq 0}}e^{-\beta t}1_{\mathcal{X}_r}(\widehat{\bm{\phi}}_{\bm{x}}(t)),
\end{equation}
where $1_{\mathcal{X}_r}(\cdot): \mathbb{R}^n \rightarrow \{0,1\}$ is the indicator function of the target set $\mathcal{X}_r$.
Obviously, $V(\bm{x})$ is bounded over the set $\overline{\mathcal{X}}$. Moreover, if $\beta=0$, 
\begin{equation}
\label{VV1111}
V(\bm{x})=\begin{cases}
    1, & \text{if~} \bm{x} \in \mathcal{RA},\\
    0, & \text{otherwise}.
  \end{cases}
\end{equation}
If $\beta>0$, 
\begin{equation}
\label{value111}
V(\bm{x})=\begin{cases}
0, &\text{if~}  \bm{x}\in \overline{\mathcal{X}}\setminus \mathcal{RA},\\
e^{-\beta \tau_{\bm{x}}}, &\text{if~} \bm{x}\in \mathcal{RA},
\end{cases}
\end{equation}
where \[\tau_{\bm{x}}=\inf\{t\in \mathbb{R}_{\geq 0}\mid \widehat{\bm{\phi}}_{\bm{x}}(t) \in \mathcal{X}_r\}\] is the first hitting time of the target set $\mathcal{X}_r$.

From \eqref{VV1111} and \eqref{value111}, we have that following proposition. 
\begin{proposition}
\label{value_reach}
When $\beta=0$, the one level set of $V(\cdot): \overline{\mathcal{X}}\rightarrow \mathbb{R}$ in \eqref{value2} is equal to the reach-avoid set $\mathcal{RA}$, i.e.,\[\{\bm{x}\in \overline{\mathcal{X}} \mid V(\bm{x})=1\}=\mathcal{RA}.\]
When $\beta>0$, the strict zero super level set of $V(\cdot): \overline{\mathcal{X}}\rightarrow \mathbb{R}$ in \eqref{value2} is equal to the reach-avoid set $\mathcal{RA}$, i.e., \[\{\bm{x}\in \overline{\mathcal{X}} \mid V(\bm{x})>0\}=\mathcal{RA}.\]
\end{proposition}

In the following we respectively obtain two sets of quantified constraints for the reach-avoid verification based on the discount value function $V(\bm{x})$ in \eqref{value2} with $\beta>0$ and $\beta=0$. Through thorough analysis on these two sets of constraints, we further obtain a set of more expressive constraints for the reach-avoid verification.

\subsection{Exponential Guidance-barrier Functions}
\label{largerzero}
In this subsection we introduce the construction of quantified constraints based on the discount value function $V(\bm{x})$ in \eqref{value2} with $\beta>0$, such that the reach-avoid verification problem is transformed into a problem of determining the existence of an exponential guidance-barrier function.

The set of constraints is derived from a system of equations admitting the value function $V(\bm{x})$ as solutions, which is formulated in Theorem \ref{equa2}.

\begin{theorem}
\label{equa2}
Given system \textbf{CSPS}, if there exists a continuously differential function $v(\bm{x}): \overline{\mathcal{X}}\rightarrow [0,1]$ such that  
\begin{align}
&\bigtriangledown_{\bm{x}}v(\bm{x}) \cdot \bm{f}(\bm{x})=\beta v(\bm{x}), \forall \bm{x}\in \overline{\mathcal{X}\setminus \mathcal{X}_r}, \label{con11}\\
&v(\bm{x})=0, \forall \bm{x}\in \partial \mathcal{X},\label{con31}\\
&v(\bm{x})=1, \forall \bm{x}\in \mathcal{X}_r,\label{con21}
\end{align}
then $v(\bm{x})=V(\bm{x})$ for $\bm{x}\in \overline{\mathcal{X}}$ and thus \[\{\bm{x}\in \overline{\mathcal{X}}\mid v(\bm{x})>0\}=\mathcal{RA},\] where $V(\cdot): \overline{\mathcal{X}}\rightarrow \mathbb{R}$ is the value function with $\beta>0$ in \eqref{value2}. 
\end{theorem}
\begin{proof}
We first consider that $\bm{x}\in \mathcal{RA}$. If $\bm{x}\in \mathcal{X}_r$, $v(\bm{x})=1=e^{0}$ according to constraint \eqref{con21}. Therefore, we just need to consider $\bm{x}\in \mathcal{RA}\setminus \mathcal{X}_r$.

From \eqref{con11}, we have that \[v(\bm{x})=e^{-\beta \tau} v(\widehat{\bm{\phi}}_{\bm{x}}(\tau)), \forall \tau \in [0,\tau_{\bm{x}}],\] where $\tau_{\bm{x}}$ is the first hitting time of the target set $\mathcal{X}_r$. Due to constraint \eqref{con21}, we further have that \[v(\bm{x})=e^{-\beta \tau_{\bm{x}}}.\]

Next, we consider that $\bm{x}\in \overline{\mathcal{X}}\setminus \mathcal{RA}$, but its resulting trajectory $\widehat{\bm{\phi}}_{\bm{x}}(\tau)$ will stay within the set $\mathcal{X}\setminus \mathcal{X}_r$ for all time. Due to constraint \eqref{con11}, we have that \[v(\bm{x})=e^{-\beta \tau} v(\widehat{\bm{\phi}}_{\bm{x}}(\tau))\] for $\tau \in \mathbb{R}_{\geq 0}.$ Since $v(\cdot): \overline{\mathcal{X}}\rightarrow \mathbb{R}$ is bounded, we have that \[v(\bm{x})=0.\] 

Finally, we consider that $\bm{x}\in \overline{\mathcal{X}}\setminus \mathcal{RA}$, but its resulting trajectory $\widehat{\bm{\phi}}_{\bm{x}}(\tau)$ will touch the set $\partial \mathcal{X}$ in finite time and never enter the target set $\mathcal{X}_r$. Let \[\tau'_{\bm{x}}=\inf\{t\in \mathbb{R}_{\geq 0}\mid \widehat{\bm{\phi}}_{\bm{x}}(t) \in \partial \mathcal{X}\}\] be the first hitting time of the set $\partial \mathcal{X}$. 

If $\bm{x}\in \partial \mathcal{X}$, $v(\bm{x})=0$ holds from constraint \eqref{con31}. Otherwise, $\tau'_{\bm{x}}>0$. Further, from \eqref{con11} we have 
\[v(\bm{x})=e^{-\beta \tau} v(\widehat{\bm{\phi}}_{\bm{x}}(\tau)), \forall \tau \in [0,\tau'_{\bm{x}}).\]
Regarding constraint \eqref{con31}, which implies that $v(\widehat{\bm{\phi}}_{\bm{x}}(\tau'_{\bm{x}}))=0$, we have $v(\bm{x})=0$.

Thus, \[v(\bm{x})=V(\bm{x})\] for $\bm{x}\in \overline{\mathcal{X}}$ according to \eqref{value111} and $\mathcal{RA}=\{\bm{x}\in \overline{\mathcal{X}}\mid v(\bm{x})>0\}$. 
\end{proof}

Via relaxing the set of equations \eqref{con11}-\eqref{con21}, we can obtain a set of inequalities for computing what we call exponential guidance-barrier function, whose existence ensures the satisfaction of the reach-avoid property in Definition \ref{RNS}. This set of inequalities is formulated in Proposition \ref{inequality1}, in which inequalities \eqref{con311} and  \eqref{con511} are obtained directly by relaxing equations \eqref{con11} and \eqref{con31}, respectively. 

\begin{proposition}
\label{inequality1}
 If there exists a continuously differentiable function $v(\bm{x}):\overline{\mathcal{X}}\rightarrow \mathbb{R}$ such that 
\begin{align}
&v(\bm{x})>0, \forall \bm{x}\in \mathcal{X}_0, \label{con3110}\\
&\bigtriangledown_{\bm{x}}v(\bm{x}) \cdot \bm{f}(\bm{x})\geq \beta v (\bm{x}), \forall \bm{x}\in \overline{\mathcal{X}\setminus \mathcal{X}_r},\label{con311} \\
&v(\bm{x})\leq 0, \forall \bm{x}\in \partial \mathcal{X},\label{con511}
\end{align}
where $\beta>0$ is a user-defined value, then the reach-avoid property in the sense of Definition \ref{RNS} holds.
\end{proposition}

Comparing constraints \eqref{con3110}-\eqref{con511} and  \eqref{0411}-\eqref{0451}, we find that the former is just a simplified version of the latter, i.e., if a function $v(\bm{x})$ satisfies constraints \eqref{0411}-\eqref{0451}, $-v(\bm{x})$ satisfies \eqref{con3110}-\eqref{con511}. The former can be obtained by removing constraints \eqref{0431} and \eqref{0441} and reversing the inequality sign in the rest of constraints in the latter. Therefore, we did not give a proof of Proposition \ref{inequality1} here. Also, due to its more concise form, constraint \eqref{con3110}-\eqref{con511} will be used for discussions and comparisons instead of  constraint \eqref{0411}-\eqref{0451} in the sequel.

If an exponential guidance-barrier function $v(\bm{x})$ satisfying constraint \eqref{con3110}-\eqref{con511} is found, the reach-avoid property in Definition \ref{RNS} holds, which is justified via Proposition \ref{inequality1}. Moreover, it is observed that the set \[\mathcal{R}=\{\bm{x}\in \overline{\mathcal{X}}\mid v(\bm{x})>0\}\] is an under-approximation of the reach-avoid set, i.e., $\mathcal{R}\subseteq \mathcal{RA}$. Also, due to constraint \eqref{con311}, we conclude that the set $\mathcal{R}$ is an invariant for system \textbf{CTolS} until it enters the target set $\mathcal{X}_r$, i.e., the boundary $\partial \mathcal{R}=\{\bm{x}\in \overline{\mathcal{X}}\mid v(\bm{x})=0\}$ is a barrier, preventing system \textbf{CTolS} from leaving the set $\mathcal{R}$ and escorting system \textbf{CTolS} to the target set $\mathcal{X}_r$ safely; furthermore, we observe that if $v(\bm{x})$ satisfies constraint \eqref{con3110}-\eqref{con511}, it must satisfy 
\begin{equation}
\label{exponential0}
\begin{split}
& \bigtriangledown_{\bm{x}}v(\bm{x})\cdot \bm{f}(\bm{x})\geq \beta v(\bm{x}), \forall \bm{x}\in \overline{\mathcal{R}\setminus \mathcal{X}_r}.
\end{split}
\end{equation}
This constraint indicates that trajectories starting from $\mathcal{R}$ will approach the target set $\mathcal{X}_r$ at an exponential rate of $\beta$. This is why we term a solution to the set of constraints \eqref{con3110}-\eqref{con511} exponential guidance-barrier function. The above analysis also uncovers a necessary condition such that $v(\bm{x})$ is a solution to constraint \eqref{con3110}-\eqref{con511}, which is \[\mathcal{R}\cap \mathcal{X}_r\neq \emptyset.\] 

The condition of entering the target set at an exponential rate is strict for many cases, limiting the application of constraint \eqref{con3110}-\eqref{con511} to the reach-avoid verification in practice. On the other hand, since the initial set $\mathcal{X}_0$ should be a subset of the set $\mathcal{R}$, the less conservative the set $\mathcal{R}$ is, the more likely the reach-avoid property is able to be verified. It is concluded from constraint \eqref{exponential0} that the smaller $\beta$ is, the less conservative the set $\mathcal{R}$ is inclined to be. This is illustrated in the following example. 
\begin{example}
\label{illu1}
Consider an academic example,
\begin{equation}
    \begin{cases}
           &\dot{x}=-0.5x-0.5y+0.5xy\\
           &\dot{y}=-0.5y+0.5
    \end{cases}
\end{equation}
with $\mathcal{X}=\{(x,y)^{\top}\mid x^2 + y^2 - 1 <0\}$, $\mathcal{X}_r=\{(x,y)^{\top}\mid (x+0.2)^2 + (y -0.7)^2 - 0.02<0\}$ and $\mathcal{X}_0=\{(x,y)^{\top}\mid 0.1-x<0,x-0.5<0,-0.8-y<0,y+0.5<0\}$.

In this example we use $\beta=0.1$ and $\beta=1$ to illustrate the effect of $\beta$ on the reach-avoid verification. Constraint \eqref{con3110}-\eqref{con511} is solved by relaxing it into a semi-definite program, which is presented in Appendix, via sum-of-squares decomposition techniques. The degree of all polynomials in the resulting semi-definite program is taken the same and is taken in order of $\{2,4,6,8,10,\ldots,20\}$. When the reach-avoid property is verified successfully, the computations will terminate. The degree is respectively $14$ for $\beta=0.1$ and $20$ for $\beta=1$ when termination. Both of the computed sets $\mathcal{R}$ are showcased in Fig. \ref{illu_eps}, which almost collide with each other. 
 
\begin{figure}[htb!]
\center
\includegraphics[width=2.5in,height=2.0in]{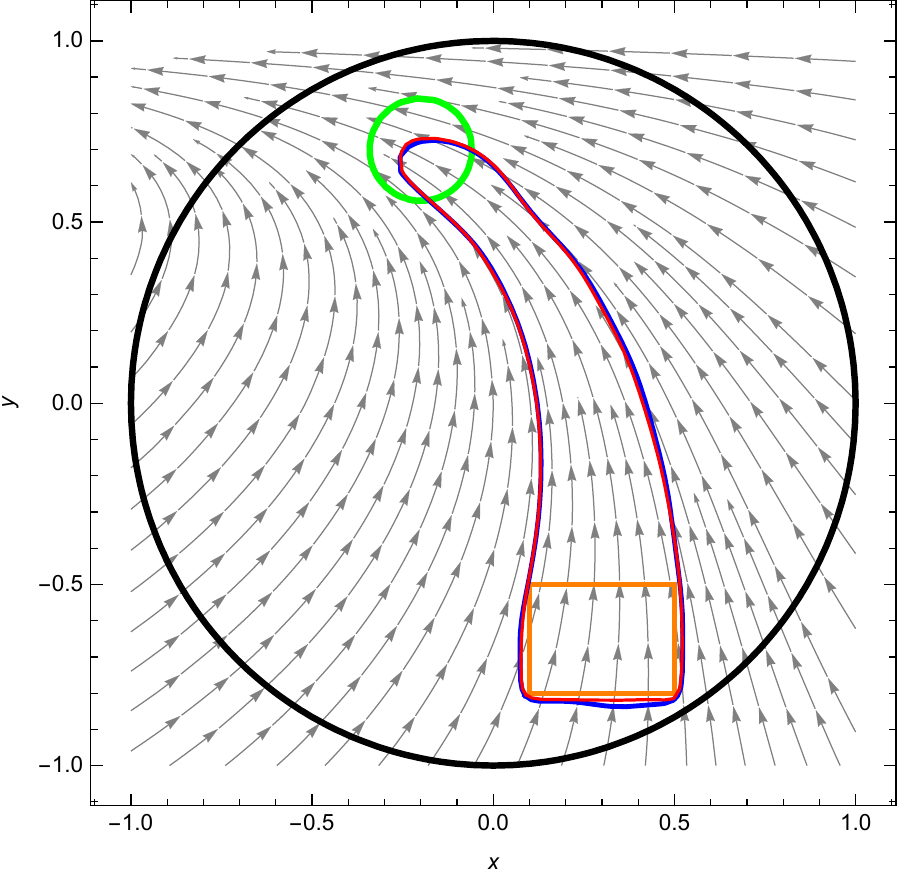} 
\caption{Green, orange and black curve- $\partial \mathcal{X}_r$, $\partial \mathcal{X}_0$ and $\partial \mathcal{X}$; blue and red curve - $\partial \mathcal{R}$, which are computed via respectively solving constraints \eqref{con3110}-\eqref{con511} with $\beta=0.1$, and \eqref{con3110}-\eqref{con511} with  $\beta=1$.}
\label{illu_eps}
\end{figure}
\end{example}

It is worth emphasizing here that although the discount factor $\beta$ can arbitrarily approach zero from above,
it cannot be zero in constraint \eqref{con3110}-\eqref{con511}, since a function $v(\bm{x})$ satisfying this constraint with $\beta=0$ cannot rule out the existence of trajectories, which start from $\mathcal{X}_0$ and stay inside $\mathcal{X}\setminus \mathcal{X}_r$ for ever. Consequently, we do not recommend the use of too small $\beta$ in practical numerical computations in order to avoid numerical issue (i.e., preventing the term $\beta V(\bm{x})$ in the right hand of constraint \eqref{con311} from becoming zero numerically due to floating point errors).

Although a set of constraints for the reach-avoid verification when $\beta=0$ cannot be obtained directly from \eqref{con3110}-\eqref{con511}, we will obtain one from the discount function \eqref{value2} with $\beta=0$ in the sequel, expecting to remedy the shortcoming of the strict requirement of exponentially entering the set $\mathcal{X}_r$ when $\beta>0$.

\subsection{Asymptotic Guidance-barrier Functions}
In this subsection we elucidate the construction of constraints for the reach-avoid verification based on the discount value function $V(\bm{x})$ in \eqref{value2} with $\beta=0$. In this case, the reach-avoid verification problem
is transformed into a problem of determining the existence of
an asymptotic guidance-barrier function. 

The set of constraints is constructed via relaxing a system of equations admitting the value function $V(\bm{x})$ in \eqref{value2} with $\beta=0$ as solutions. These equations are presented in Theorem \ref{equa1}.
\begin{theorem}
\label{equa1}
Given system \textbf{CSPS}, if there exist continuously differential functions $v(\bm{x}): \overline{\mathcal{X}}\rightarrow \mathbb{R}$ and $w(\bm{x}):\overline{\mathcal{X}}\rightarrow \mathbb{R}$ satisfying 
\begin{align}
&\bigtriangledown_{\bm{x}}v(\bm{x}) \cdot \bm{f}(\bm{x})=0, \forall \bm{x}\in \overline{\mathcal{X}\setminus \mathcal{X}_r},\label{con1}\\
&v(\bm{x})=\bigtriangledown_{\bm{x}}w(\bm{x}) \cdot \bm{f}(\bm{x}), \forall \bm{x} \in \overline{\mathcal{X}\setminus \mathcal{X}_r}, \label{con2}\\
&v(\bm{x})=0, \forall \bm{x} \in \partial \mathcal{X}, \label{con4}\\
&v(\bm{x})=1, \forall \bm{x} \in \mathcal{X}_r, \label{con3}
\end{align}
then $v(\bm{x})=V(\bm{x})$ for $\bm{x}\in \overline{\mathcal{X}}$ and thus \[\{\bm{x}\in \overline{\mathcal{X}}\mid v(\bm{x})=1\}=\mathcal{RA},\] where $V(\cdot): \overline{\mathcal{X}}\rightarrow \mathbb{R}$ is the value function with $\beta=0$ in \eqref{value2}. 
\end{theorem}
\begin{proof}
From \eqref{con1}, we have that 
\begin{equation}
\label{con1_equal}
v(\bm{x})=v(\widehat{\bm{\phi}}_{\bm{x}}(\tau)), \forall  \tau \in [0,\tau_{\bm{x}}],
\end{equation}
where $\tau_{\bm{x}}\in \mathbb{R}_{\geq 0}$ is the time instant such that \[\widehat{\bm{\phi}}_{\bm{x}}(\tau) \in \overline{\mathcal{X}\setminus \mathcal{X}_r}, \forall \tau \in [0,\tau_{\bm{x}}].\] 

For $\bm{x}\in \mathcal{RA}$, we can obtain that $v(\bm{x})=1$ due to constraints \eqref{con3} and \eqref{con1_equal}.

In the following we consider $\bm{x}\in \overline{\mathcal{X}}\setminus \mathcal{RA}$.

We first consider $\bm{x} \in \overline{\mathcal{X}}\setminus \mathcal{RA}$, but its trajectory $\widehat{\bm{\phi}}_{\bm{x}}(\cdot): \mathbb{R}_{\geq 0} \rightarrow \mathbb{R}^n$ stays within the set $\mathcal{X}\setminus \mathcal{X}_r$. From \eqref{con2}, we have that \[v(\widehat{\bm{\phi}}_{\bm{x}}(\tau) )=\bigtriangledown_{\bm{y}}w(\bm{y}) \cdot \bm{f}(\bm{y})\mid_{\bm{y}=\widehat{\bm{\phi}}_{\bm{x}}(\tau)}\] for $\tau\in \mathbb{R}_{\geq 0}$.
Thus, we have that \[\int_{0}^t v(\widehat{\bm{\phi}}_{\bm{x}}(\tau)) d\tau=\int_0^t \bigtriangledown_{\bm{y}}w(\bm{y}) \cdot \bm{f}(\bm{y})\mid_{\bm{y}=\widehat{\bm{\phi}}_{\bm{x}}(\tau)} d\tau\] for $t\in \mathbb{R}_{\geq 0}$ and further \[v(\bm{x})=\frac{w(\widehat{\bm{\phi}}_{\bm{x}}(t) )-w(\bm{x})}{t}\] for $t\in \mathbb{R}_{\geq 0}$.
Since $w(\bm{x})$ is continuously differentiable function over $\overline{\mathcal{X}}$, it is bounded over $\bm{x}\in \overline{\mathcal{X}}$. Consequently, $v(\bm{x})=0$.

Next, we consider $\bm{x} \in \overline{\mathcal{X}}\setminus \mathcal{RA}$, but its trajectory $\widehat{\bm{\phi}}_{\bm{x}}(\tau)$ will touch  $\partial \mathcal{X}$ in finite time and never enters the target set $\mathcal{X}_r$. For such $\bm{x}$, we can obtain that $v(\bm{x})=0$ due to constraints \eqref{con1_equal} and \eqref{con4}.

Thus, according to \eqref{VV1111}, \[v(\bm{x})=V(\bm{x})\] for $\bm{x}\in \overline{\mathcal{X}}$. Further, from Lemma \ref{value_reach}, \[\{\bm{x}\in \overline{\mathcal{X}} \mid v(\bm{x})=1\}=\mathcal{RA}\] holds. 
\end{proof}

Based on the system of equations \eqref{con1}-\eqref{con3}, we have a set of inequalities as shown in Proposition \ref{inequality} for computing an asymptotic guidance-barrier function $v(\bm{x})$ to ensure the satisfaction of reach-avoid properties in the sense of Definition \ref{RNS}. In Proposition \ref{inequality}, inequalities \eqref{con3_0}, \eqref{con4_0} and \eqref{con5_0} are obtained directly by relaxing equations \eqref{con1}, \eqref{con2} and \eqref{con4}, respectively. 
\begin{proposition}
\label{inequality}
 If there exist a  continuously differentiable function $v(\bm{x}):\overline{\mathcal{X}}\rightarrow \mathbb{R}$ and a continuously differentiable function $w(\bm{x}):\overline{\mathcal{X}}\rightarrow \mathbb{R}$ satisfying 
\begin{align}
& v(\bm{x})>0, \forall \bm{x}\in \mathcal{X}_0, \label{con4_00}\\
&\bigtriangledown_{\bm{x}}v(\bm{x}) \cdot \bm{f}(\bm{x})\geq 0, \forall \bm{x}\in \overline{\mathcal{X}\setminus \mathcal{X}_r}, \label{con3_0} \\
&v(\bm{x})-\bigtriangledown_{\bm{x}}w(\bm{x}) \cdot \bm{f}(\bm{x})\leq 0, \forall \bm{x}\in \overline{\mathcal{X}\setminus \mathcal{X}_r},\label{con4_0} \\
&v(\bm{x})\leq 0, \forall \bm{x}\in \partial \mathcal{X},\label{con5_0}
\end{align}
then the reach-avoid property in the sense of Definition \ref{RNS} holds.
\end{proposition}
\begin{proof}
Let $\mathcal{R}=\{\bm{x}\in \overline{\mathcal{X}}\mid v(\bm{x})>0\}$. We will show that 
\[\mathcal{R}\subseteq \mathcal{RA}.\] If it holds, we can obtain the conclusion since \[\mathcal{X}_0\subseteq \mathcal{RA},\] which is obtained from constraint \eqref{con4_00}.

Let $\bm{x}_0\in \mathcal{R}$. Obviously, $\bm{x}_0\in \mathcal{X}$ due to constraint \eqref{con5_0}. If $\bm{x}_0\in \mathcal{X}_{r}$, $\bm{x}_0\in \mathcal{RA}$ holds obviously. Therefore, in the following we assume $\bm{x}_0\in \mathcal{R} \setminus  \mathcal{X}_{r}$. We will prove that there exists $t\in \mathbb{R}_{\geq 0}$ satisfying \[\bm{\phi}_{\bm{x}_0}(t) \in \mathcal{X}_{r}\bigwedge \forall \tau \in [0,t]. \bm{\phi}_{\bm{x}_0}(\tau) \in \mathcal{R}.\]

Assume that there exists $t\in \mathbb{R}_{\geq 0}$ such that \[\bm{\phi}_{\bm{x}_0}(t) \in \partial \mathcal{R}\bigwedge \forall \tau\in [0,t). \bm{\phi}_{\bm{x}_0}(\tau) \in \mathcal{R} \setminus  \mathcal{X}_{r}.\] From \eqref{con3_0}, we have that \[v(\bm{\phi}_{\bm{x}_0}(\tau))\geq v(\bm{x}_0)>0\] for $\tau \in [0,t]$, contradicting that $v(\bm{x})= 0$ for $\bm{x}\in \partial \mathcal{R}.$

Therefore, either 
\begin{equation}
\label{alwa0}
\bm{\phi}_{\bm{x}_0}(\tau) \in \mathcal{R} \setminus  \mathcal{X}_{r}, \forall \tau \in \mathbb{R}_{\geq 0}
\end{equation}
 or 
\begin{equation}
\label{alwa10}
\exists t \in  \mathbb{R}_{\geq 0}. \bm{\phi}_{\bm{x}_0}(t)\in \mathcal{X}_{r} \bigwedge \forall \tau \in [0,t].\bm{\phi}_{\bm{x}_0}(\tau) \in \mathcal{R}
\end{equation}
 holds.

Assume that \eqref{alwa0} holds. 
From \eqref{con3_0}, we have that
\begin{equation}
\label{mono10}
v(\bm{\phi}_{\bm{x}_0}(\tau))\geq v(\bm{x}_0)>0, \forall \tau \in \mathbb{R}_{\geq 0}.
\end{equation}

From \eqref{con4_0}, we have that for $\tau \in \mathbb{R}_{\geq 0}$,
\begin{equation}
\label{ineq10}
\begin{split}
v(\bm{\phi}_{\bm{x}_0}(\tau))\leq \bigtriangledown_{\bm{x}}w(\bm{x}) \cdot \bm{f}(\bm{x})\mid_{\bm{x}=\bm{\phi}_{\bm{x}_0}(\tau)}.
\end{split}
\end{equation}
 Thus, 
 \[w(\bm{\phi}_{\bm{x}_0}(t))-w(\bm{x}_0)\geq t v(\bm{x}_0), \forall t\in \mathbb{R}_{\geq 0}, \]
 implying that \[\lim_{t\rightarrow +\infty}w(\bm{\phi}_{\bm{x}_0}(t))=+\infty.\] 
 Since $w(\cdot):\overline{\mathcal{X}}\rightarrow \mathbb{R}$ is continuously differentiable and $\overline{\mathcal{X}}$ is compact, $w(\cdot):\overline{\mathcal{X}}\rightarrow \mathbb{R}$ is bounded. This is a contradiction. Thus, \eqref{alwa10} holds, implying that $\bm{x} \in \mathcal{RA}$ and thus $\mathcal{R}\subseteq \mathcal{RA}$.
\end{proof}

According to Proposition \ref{inequality}, the reach-avoid property in the sense of Definition \ref{RNS} can be verified via searching for a feasible solution $v(\bm{x})$ to the set of constraints \eqref{con4_00}-\eqref{con5_0}. Also, from the proof of Proposition \ref{inequality}, we find that the set $\mathcal{R}=\{\bm{x}\in \overline{\mathcal{X}}\mid v(\bm{x})>0\}$ will also be an invariant for system \textbf{CTolS} until it enters the target set $\mathcal{X}_r$, and $\mathcal{R}\cap \mathcal{X}_r\neq \emptyset$. Aspired by the notion of asymptotic stability in stability analysis, if system \textbf{CTols} starting from the set $\mathcal{R}=\{\bm{x}\in \overline{\mathcal{X}}\mid v(\bm{x})>0\}$ with $v(\bm{x})$ satisfying constraint \eqref{con3110}-\eqref{con511}, enters the target set $\mathcal{X}_r$ at an exponential rate, then system \textbf{CTols} starting from the set $\mathcal{R}=\{\bm{x}\in \overline{\mathcal{X}}\mid v(\bm{x})>0\}$ with $v(\bm{x})$ satisfying constraint \eqref{con4_00}-\eqref{con5_0}, would enter the target set $\mathcal{X}_r$ in an asymptotic sense. This is the reason that we term a function $v(\bm{x})$ satisfying constraint \eqref{con4_00}-\eqref{con5_0} asymptotic guidance-barrier function.

Comparing constraints \eqref{con3110}-\eqref{con511} and \eqref{con4_00}-\eqref{con5_0}, it is easy to find that an auxillary function $w(\bm{x})$ is introduced when $\beta=0$ in constraint \eqref{con311}. This constraint having $w(\bm{x})$, i.e., \eqref{con4_0}, excludes trajectories starting from $\mathcal{X}_0$ and staying inside $\mathcal{X}\setminus \mathcal{X}_r$ for ever. On the other hand, constraint \eqref{con3_0} keeps trajectories, starting from $\mathcal{R}$, inside it before entering the target set $\mathcal{X}_r$. Also, we observe that constraint \eqref{con3110}-\eqref{con511} can be derived from the set of constraints \eqref{con4_00}-\eqref{con5_0} via taking $w(\bm{x})=\frac{1}{\beta}v(\bm{x})$. In such circumstances, constraint \eqref{con3_0} is redundant and should be removed. Further, if there exists a function 
$v(\bm{x})$ satisfying constraint \eqref{exponential0}, there must exist a function $w(\bm{x})$, which can take $\frac{1}{\beta}v(\bm{x})$ for instance, such that 
\begin{equation}
\label{asymp}
\begin{split}
& \bigtriangledown_{\bm{x}}v(\bm{x})\cdot \bm{f}(\bm{x})\geq 0, \forall \bm{x}\in \overline{\mathcal{R}\setminus \mathcal{X}_r},\\
&v(\bm{x})-\bigtriangledown_{\bm{x}}w(\bm{x})\cdot \bm{f}(\bm{x})\leq 0, \forall \bm{x}\in \overline{\mathcal{R}\setminus \mathcal{X}_r}
\end{split}
\end{equation}
holds. Thus, constraint \eqref{asymp} is more expressive than \eqref{exponential0} and consequently is more likely to produce less conservative set $\mathcal{R}$. We in the following continue to use the scenario in Example \ref{illu1} to illustrate this.

\begin{example}
\label{illu2}
Consider the scenario in Example \ref{illu1} again. Similarly, we also solve constraint \eqref{con4_00}-\eqref{con5_0} via relaxing it into a semi-definite program, which is shown in Appendix. The reach-avoid property is verified when polynomials of degree $12$ in the resulting semi-definite program are taken. The computed $\mathcal{R}$ is showcased in Fig. \ref{illu_eps1}.
 
\begin{figure}[htb!]
\center
\includegraphics[width=2.5in,height=2.0in]{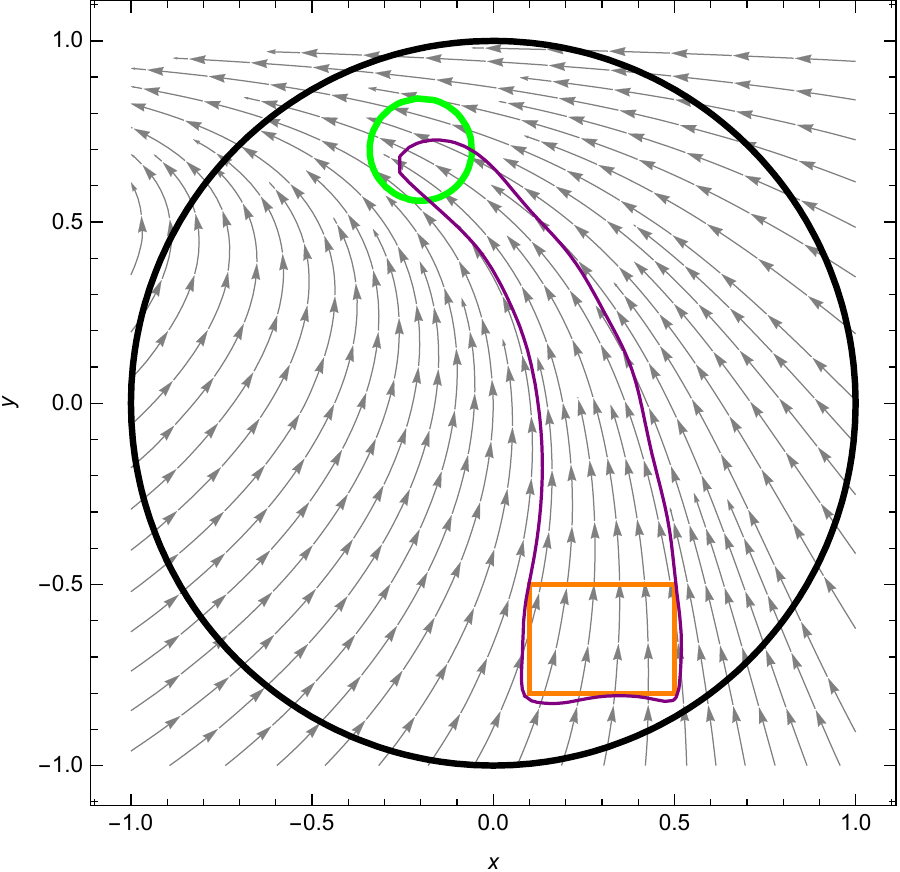} 
\caption{Green, orange and black curve- $\partial \mathcal{X}_r$, $\partial \mathcal{X}_0$ and $\partial \mathcal{X}$; purple curve - $\partial\mathcal{R}$ computed via solving constraint \eqref{con4_00}-\eqref{con5_0}.}
\label{illu_eps1}
\end{figure}
\end{example}

\begin{remark}
As done in verifying invariance of a set using barrier certificate methods \cite{ames2019control}, one simple application scenario, reflecting the advantage of constraint \eqref{asymp} over \eqref{exponential0} further, is on verifying whether trajectories starting from a given open set $\widehat{\mathcal{R}}$, which may be designed a priori via the Monte-Carlo simulation method, will enter the target set $\mathcal{X}_r$ eventually while staying inside it before the first target hit, where \[\widehat{\mathcal{R}}=\{\bm{x}\in \mathbb{R}^n\mid \widehat{v}(\bm{x})>0\}\] with $\widehat{v}(\bm{x})$ being continuously differentiable and $\mathcal{X}_r\cap \widehat{\mathcal{R}}\neq \emptyset$. We are inclined to verifying whether there exists a continuously differentiable function $w(\bm{x})$ satisfying
\begin{equation*}
\begin{split}
& \bigtriangledown_{\bm{x}}\widehat{v}(\bm{x})\cdot \bm{f}(\bm{x})\geq 0, \forall \bm{x}\in \overline{\widehat{\mathcal{R}}\setminus \mathcal{X}_r},\\
&\widehat{v}(\bm{x})-\bigtriangledown_{\bm{x}}w(\bm{x})\cdot \bm{f}(\bm{x})\leq 0, \forall \bm{x}\in \overline{\widehat{\mathcal{R}}\setminus \mathcal{X}_r},
\end{split}
\end{equation*}
 instead of verifying whether there exists $\beta>0$ such that
\begin{equation*}
\begin{split}
\bigtriangledown_{\bm{x}}\widehat{v}(\bm{x})\cdot \bm{f}(\bm{x})\geq \beta \widehat{v}(\bm{x}), \forall \bm{x}\in \overline{\widehat{\mathcal{R}}\setminus \mathcal{X}_r},
\end{split}
\end{equation*}
because the former is more expressive than the latter.
\end{remark}

Although there are some benefits on the use of constraints \eqref{con4_00}-\eqref{con5_0} over \eqref{con3110}-\eqref{con511} for the reach-avoid verification, there is still a defect caused by constraint \eqref{con3_0}, possibly limiting the application of constraint \eqref{con4_00}-\eqref{con5_0} to some extent. Unlike constraint \eqref{con311} in \eqref{con3110}-\eqref{con511}, constraint \eqref{con3_0} not only requires the Lie derivative of function $v(\bm{x})$ along the flow of system \textbf{CTolS} to be non-negative over the set $\overline{\mathcal{R}\setminus \mathcal{X}_r}$, but also over $\overline{\mathcal{X}\setminus (\mathcal{X}_r\cup \mathcal{R}})$. One simple solution to remedy this defect is to combine constraints \eqref{con4_00}-\eqref{con5_0} and \eqref{con3110}-\eqref{con511} together, and obtain a set of constraints which is more expressive. These constraints are presented in Proposition \ref{inequality_new}.
\begin{proposition}
\label{inequality_new}
 If there exist continuously differentiable functions $v_1(\bm{x}), v_2(\bm{x}) :\overline{\mathcal{X}}\rightarrow \mathbb{R}$ and $w(\bm{x}):\overline{\mathcal{X}}\rightarrow \mathbb{R}$ satisfying 
\begin{align}
& v_1(\bm{x})+v_2(\bm{x})>0, \forall \bm{x}\in \mathcal{X}_0, \label{con4_000}\\
&\bigtriangledown_{\bm{x}}v_1(\bm{x}) \cdot \bm{f}(\bm{x})\geq 0, \forall \bm{x}\in \overline{\mathcal{X}\setminus \mathcal{X}_r}, \label{con3_00} \\
&v_1(\bm{x})-\bigtriangledown_{\bm{x}}w(\bm{x}) \cdot \bm{f}(\bm{x})\leq 0, \forall \bm{x}\in \overline{\mathcal{X}\setminus \mathcal{X}_r},\label{con4_0_0} \\
&v_1(\bm{x})\leq 0, \forall \bm{x}\in \partial \mathcal{X},\label{con5_00}\\
&\bigtriangledown_{\bm{x}}v_2(\bm{x}) \cdot \bm{f}(\bm{x})\geq \beta v_2(\bm{x}), \forall \bm{x}\in \overline{\mathcal{X}\setminus \mathcal{X}_r}, \label{con6_00} \\
&v_2(\bm{x})\leq 0, \forall \bm{x}\in \partial \mathcal{X},\label{con7_00}
\end{align}
where $\beta\in (0,+\infty)$, then the reach-avoid property in the sense of Definition \ref{RNS} holds.
\end{proposition}
\begin{proof}
Let $\bm{x}_0\in \mathcal{R}=\{\bm{x}\in \overline{\mathcal{X}}\mid v_1(\bm{x})+v_2(\bm{x})>0\}$, we have that $\bm{x}_0\in \{\bm{x}\in \overline{\mathcal{X}}\mid v_1(\bm{x})>0\}$ or $\bm{x}_0\in \{\bm{x}\in \overline{\mathcal{X}}\mid v_2(\bm{x})>0\}$. Following Proposition \ref{inequality1} and \ref{inequality}, we have the conclusion.
\end{proof}

Due to constraint \eqref{con4_000}, $v_1(\bm{x})$ may not be an asymptotic guidance-barrier function satisfying constraint \eqref{con4_00}-\eqref{con5_0}. Similarly,  $v_2(\bm{x})$ may not be an exponential guidance-barrier function satisfying constraint \eqref{con3110}-\eqref{con511}. Thus, constraint \eqref{con4_000}-\eqref{con7_00} is weaker than both of constraints \eqref{con4_00}-\eqref{con5_0} and \eqref{con3110}-\eqref{con511}. The set $\mathcal{R}=\{\bm{x}\in \overline{\mathcal{X}}\mid v_1(\bm{x})+v_2(\bm{x})>0\}$ is a mix of states  entering the target set $\mathcal{X}_r$ at an exponential rate and ones entering the target set $\mathcal{X}_r$ at an asymptotic rate, thus we term $v_1(\bm{x})+v_2(\bm{x})$ asymptotic guidance-barrier function. However, we cannot guarantee that the set $\mathcal{R}$ still satisfies $\mathcal{R}\cap \mathcal{X}_r\neq \emptyset$ and it is an invariant for system \textbf{CTolS} until it enters the target set $\mathcal{X}_r$. If an initial state $\bm{x}_0 \in \mathcal{X}_0$ is a state such that $v_2(\bm{x})>0$, then the trajectory starting from it will stay inside the set $\mathcal{R}$ until it enters the target set $\mathcal{X}_r$, since 
\begin{equation*}
    \begin{split}
&\frac{d (v_1+v_2)}{d t}=\bigtriangledown_{\bm{x}}v_1(\bm{x}) \cdot \bm{f}(\bm{x})+\bigtriangledown_{\bm{x}}v_2(\bm{x}) \cdot \bm{f}(\bm{x})\\
&\geq \beta v_2(\bm{x}), \forall \bm{x}\in \overline{\mathcal{X}\setminus \mathcal{X}_r}
\end{split}
\end{equation*}
holds; otherwise, we cannot have such a conclusion.  Instead, we have that $\mathcal{R}_i=\{\bm{x}\in \overline{\mathcal{X}}\mid v_i(\bm{x})>0\}$ satisfies $\mathcal{R}_i\cap \mathcal{X}_r\neq \emptyset$ and is an invariant for system \textbf{CTolS} until it enters the target set $\mathcal{X}_r$, if $\mathcal{R}_i\neq \emptyset$, where $i\in\{1,2\}$. Let's illustrate this via an example. 

\begin{example}
\label{illu4}
Consider the scenario in Example \ref{illu1} again. Similarly, we  solve constraint \eqref{con4_000}-\eqref{con7_00} with $\beta=2$ via relaxing it into a semi-definite program, which is shown in Appendix. The reach-avoid property is verified when polynomials of degree $12$ are taken. The computed $\mathcal{R}$ is shown in Fig. \ref{illu_eps3}. For this case, the set $\mathcal{R}_2=\{\bm{x}\in \overline{\mathcal{X}}\mid v_2(\bm{x})>0\}$ is empty. It is observed from Fig. \ref{illu_eps3} that the set $\mathcal{R}$ does not intersect $\mathcal{X}_r$, and system \textbf{CTolS} leaves it before entering the target set $\mathcal{X}_r$. However, the set $\mathcal{R}_1$ is an invariant for system \textbf{CTolS} until it enters the target set $\mathcal{X}_r$ and $\mathcal{R}_1\cap \mathcal{X}_r\neq \emptyset$.

\begin{figure}[htb!]
\center
\includegraphics[width=2.5in,height=2.0in]{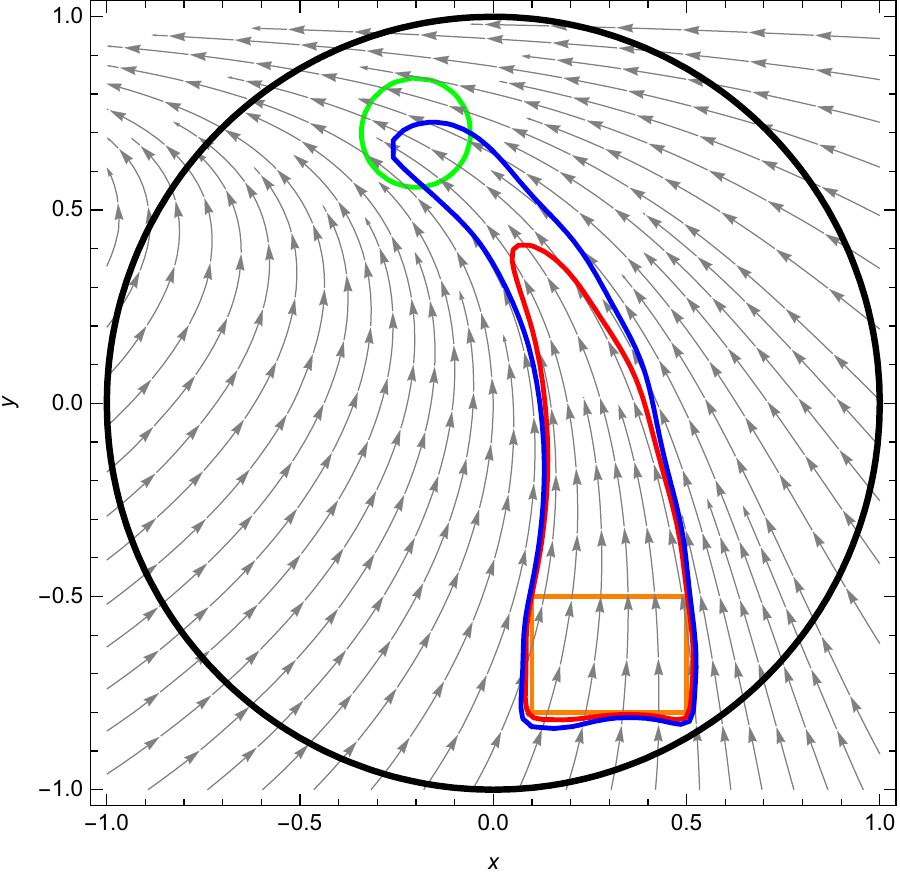} 
\caption{Green, orange and black curve- $\partial \mathcal{X}_r$, $\partial \mathcal{X}_0$ and $\partial \mathcal{X}$; blue and red curve - $\partial\mathcal{R}_1$ and $\partial\mathcal{R}$   computed via solving constraint \eqref{con4_000}-\eqref{con7_00}.}
\label{illu_eps3}
\end{figure}
\end{example}

The other more sophisticated solution of enhancing constraint \eqref{con4_00}-\eqref{con5_0} is to replace constraint \eqref{con3_0} with \[\bigtriangledown_{\bm{x}}v(\bm{x}) \cdot \bm{f}(\bm{x})\geq \alpha(\bm{x}), \forall \bm{x}\in \overline{\mathcal{X}\setminus \mathcal{X}_r},\] where $\alpha(\cdot):\overline{\mathcal{X}}\rightarrow \mathbb{R}$ is a continuous function satisfying $\alpha(\bm{x})\geq 0$ over $\{\bm{x} \in \overline{\mathcal{X}\setminus \mathcal{X}_r} \mid v(\bm{x})\geq 0\}$. One instance for $\alpha(\bm{x})$ is $\beta(\bm{x})v(\bm{x})$, where $\beta(\cdot): \overline{\mathcal{X}}\rightarrow [0,+\infty)$. The new constraints are formulated in Proposition \ref{new}.
\begin{proposition}
\label{new}
If there exists  a continuously differentiable function $v(\bm{x}):\overline{\mathcal{X}}\rightarrow \mathbb{R}$, a continuous function $\alpha(\cdot):\mathbb{R}^n\rightarrow \mathbb{R}$ satisfying $\alpha(\bm{x})\geq 0$ over $\{\bm{x} \in \overline{\mathcal{X}\setminus \mathcal{X}_r} \mid v(\bm{x})\geq 0\}$, and a continuously differentiable function $w(\bm{x}):\overline{\mathcal{X}}\rightarrow \mathbb{R}$ satisfying 
\begin{align}
& v(\bm{x})>0, \forall \bm{x}\in \mathcal{X}_0, \label{con4_00_new}\\
&\bigtriangledown_{\bm{x}}v(\bm{x}) \cdot \bm{f}(\bm{x})\geq \alpha(\bm{x}), \forall \bm{x}\in \overline{\mathcal{X}\setminus \mathcal{X}_r}, \label{con3_0_new} \\
&v(\bm{x})-\bigtriangledown_{\bm{x}}w(\bm{x}) \cdot \bm{f}(\bm{x})\leq 0, \forall \bm{x}\in \overline{\mathcal{X}\setminus \mathcal{X}_r},\label{con4_0_new} \\
&v(\bm{x})\leq 0, \forall \bm{x}\in \partial \mathcal{X},\label{con5_0_new}
\end{align}
then the reach-avoid property in the sense of Definition \ref{RNS} holds.
\end{proposition}
\begin{proof}
Let $\mathcal{R}=\{\bm{x}\in \overline{\mathcal{X}}\mid v(\bm{x})>0\}$. 
From constraint \eqref{con3_0_new}, we have  that if $\bm{\phi}_{\bm{x}_0} (t) \in \overline{\mathcal{R}\setminus \mathcal{X}_r}$, \[\bigtriangledown_{\bm{x}}v(\bm{x}) \cdot \bm{f}(\bm{x})\mid_{\bm{x}=\bm{\phi}_{\bm{x}_0} (t)}\geq 0.\] Then, following the arguments in the proof of Proposition \ref{inequality}, we have the conclusion. 
\end{proof}

Constraint \eqref{con4_00_new}-\eqref{con5_0_new} is less strict than constraint \eqref{con4_00}-\eqref{con5_0} in that the former only requires the Lie derivative of $v(\bm{x})$ along the flow of system \textbf{CTolS} to be non-negative over the set $\overline{\mathcal{R}\setminus \mathcal{X}_r}$ rather than $\overline{\mathcal{X}\setminus \mathcal{X}_r}$, due to constraint \eqref{con3_0_new}. Constraint \eqref{con4_00_new}-\eqref{con5_0_new} does not impose any restrictions on the Lie derivative of $v(\bm{x})$ along the flow of system \textbf{CTolS} over the set $\overline{\mathcal{X}\setminus \mathcal{R}}$. Moreover, it is more expressive since it degenerates to constraint \eqref{con4_00}-\eqref{con5_0} when $\alpha(\cdot)\equiv 0$, and it is more expressive than constraint \eqref{con3110}-\eqref{con511}, since the former degenerates to the latter when \[\alpha(\bm{x})=\beta v(\bm{x})\] and \[w(\bm{x})=\frac{1}{\beta} v(\bm{x}).\] 
Besides, the set $\mathcal{R}$ obtained via solving constraint \eqref{con4_00_new}-\eqref{con5_0_new} will be an invariant for system \textbf{CTolS} until it enters the target set $\mathcal{X}_r$, and $\mathcal{R}\cap \mathcal{X}_r\neq \emptyset$.

\begin{example}
\label{illu4}
Consider the system in Example \ref{illu1} 
with $\mathcal{X}=\{(x,y)^{\top}\mid x^2 + y^2 - 1 <0\}$, $\mathcal{X}_r=\{(x,y)^{\top}\mid (x+0.2)^2 + (y -0.7)^2 - 0.02<0\}$ and $\mathcal{X}_0=\{(x,y)^{\top}\mid 0.1-x<0,x-0.5<0,-0.8-y<0,y+0.4<0\}$.

In this example we use $\alpha(v(\bm{x}))=x^2 v(\bm{x})$ to illustrate the benefits of constraints \eqref{con4_00_new}-\eqref{con5_0_new} on the reach-avoid verification. The constraint \eqref{con4_00_new}-\eqref{con5_0_new} is solved by relaxing it into a semi-definite program which is presented in Appendix. The degree of all polynomials in the resulting semi-definite program is taken the same and in order of $\{2,4,6,8,10,\ldots,20\}$. When the reach-avoid property is verified successfully, the computations terminate. The degree is $12$ for termination, and the computed set $\mathcal{R}$ is showcased in Fig. \ref{illu_eps4}. In contrast, the degree is $14$ when using constraints \eqref{con3110}-\eqref{con511} with $\beta=0.1$ and \eqref{con4_00}-\eqref{con5_0} to verify the reach-avoid property. Consequently, these experiments further support our theoretical analysis  that constraint \eqref{con4_00_new}-\eqref{con5_0_new} is more expressive and provides more chances for verifying the reach-avoid property successfully.
 
\begin{figure}[htb!]
\center
\includegraphics[width=2.5in,height=2.0in]{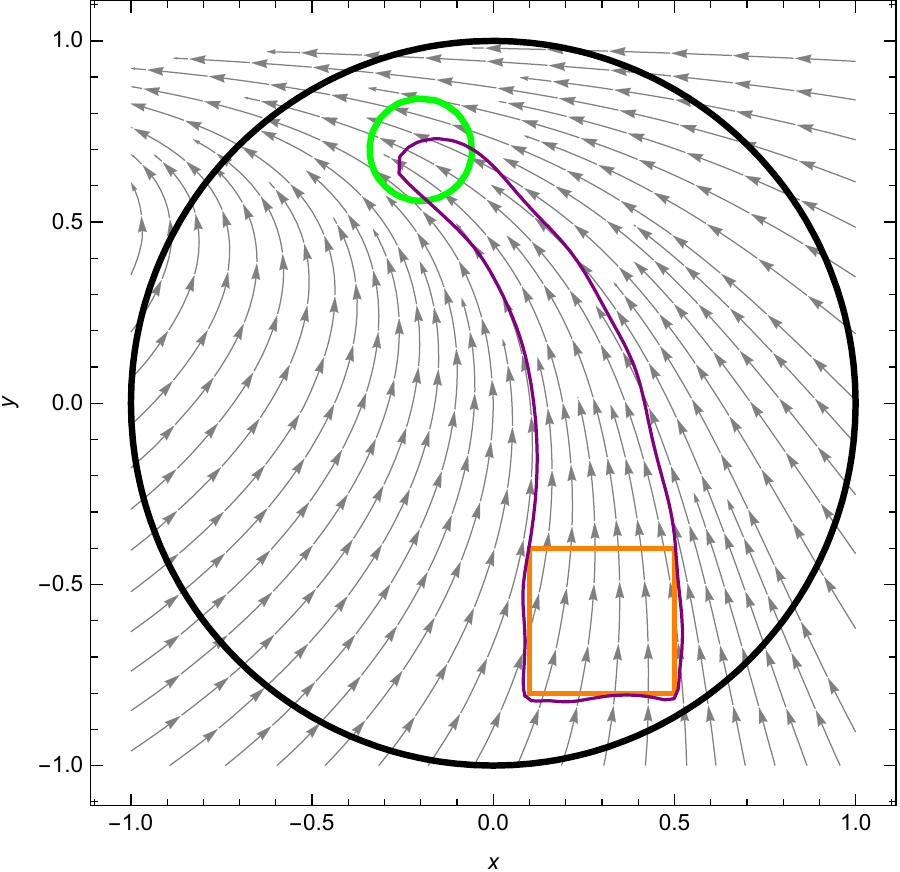} 
\caption{Green, orange and black curve- $\partial \mathcal{X}_r$, $\partial \mathcal{X}_0$ and $\partial \mathcal{X}$; purple curve - $\partial\mathcal{R}$,which is computed via solving constraint \eqref{con4_00_new}-\eqref{con5_0_new}.}
\label{illu_eps4}
\end{figure}

\end{example}

\section{Experiments}
\label{ex}
We further demonstrate the theoretical development and performance of the proposed methods on several examples, i.e., Examples \ref{ex1}-\ref{dubin}. All of constraints \eqref{041}-\eqref{043}, \eqref{con3110}-\eqref{con511}, \eqref{con4_00}-\eqref{con5_0}, \eqref{con4_000}-\eqref{con7_00} and \eqref{con4_00_new}-\eqref{con5_0_new} are addressed via encoding them into semi-definite programs. The formulated semi-definite programs are presented in Appendix. In the computations, the degree of unknown polynomials in the resulting semi-definite programs is taken the same and in order of $\{2,4,6,8,10,\ldots,20\}$. When the reach-avoid property is verified successfully, the computations terminate. A return of `Successfully solved (MOSEK)' from YALMIP will denote that a feasible solution is found, and the reach-avoid property is successfully verified.

\begin{example}
\label{ex1}
Consider the scenario in Example \ref{illu4}. As analyzed in Subsection \ref{largerzero}, the smaller the discount factor $\beta$ is in constraint  \eqref{con3110}-\eqref{con511}, the more likely the reach-avoid property is able to be verified. Thus, in this example, we supplement some experiments involving constraint  \eqref{con3110}-\eqref{con511} with $\beta<0.1$ for more comprehensive and fair comparisons with the proposed methods in the present work. In these experiments, $\beta=10^{-2},10^{-3},10^{-4},10^{-5},10^{-6}$ are used. For all of these experiments, the computations terminate when the degree takes $14$. All the results, including the ones in Example \ref{illu4}, further validate the benefits of constraint \eqref{con4_00_new}-\eqref{con5_0_new} over constraints \eqref{con4_00}-\eqref{con5_0} and \eqref{con3110}-\eqref{con511} in terms of stronger expressiveness. 

We also experimented using constraint \eqref{041}-\eqref{043}, and the reach-avoid property is verified when the degree is 14. 
\end{example}

\begin{example}[Van der Pol Oscillator]
\label{ex2}
Consider the reversed-time Van der Pol oscillator given by 
\begin{equation*}
\begin{cases}
&\dot{x}=-2y\\
&\dot{y}=0.8x+10(x^2-0.21)y
\end{cases}
\end{equation*}
with $\mathcal{X}=\{(x,y)^{\top}\mid x^2+y^2-1<0\}$,  $\mathcal{X}_0=\{(x,y)^{\top}\mid -x<-0.6,x<0.8,-y<0,y<0.2\}$ and $\mathcal{X}_r=\{(x,y)^{\top}\mid x^2+y^2-0.01<0\}$.

 The reach-avoid property is verified when the degree is 8, 8 and 12 for constraints \eqref{041}-\eqref{043}, \eqref{con4_00}-\eqref{con5_0} and \eqref{con3110}-\eqref{con511} with $\beta=0.1$, respectively. We did not obtain any positive verification result from constraint \eqref{con3110}-\eqref{con511} with $\beta=1$. However, it can be improved by solving constraint \eqref{con4_000}-\eqref{con7_00} with $\beta=1$ and the reach-avoid property is verified when the degree is 8. Further, if constraint \eqref{con4_00_new}-\eqref{con5_0_new} is used with $\alpha(\bm{x})=x^2 v(\bm{x})$, the degree is 6. Some of the computed sets $\mathcal{R}$ are illustrated in Fig. \ref{fig_2}.

Besides, we also experimented using constraints  \eqref{con3110}-\eqref{con511} and \eqref{con4_000}-\eqref{con7_00} with $\beta=10^{-2},10^{-3},10^{-4},10^{-5},10^{-6}$ for more comprehensive and fair comparisons. The computations terminate when the degree takes $8$ for all of these experiments.

\begin{figure}[htb!]
\center
\includegraphics[width=2.5in,height=2.0in]{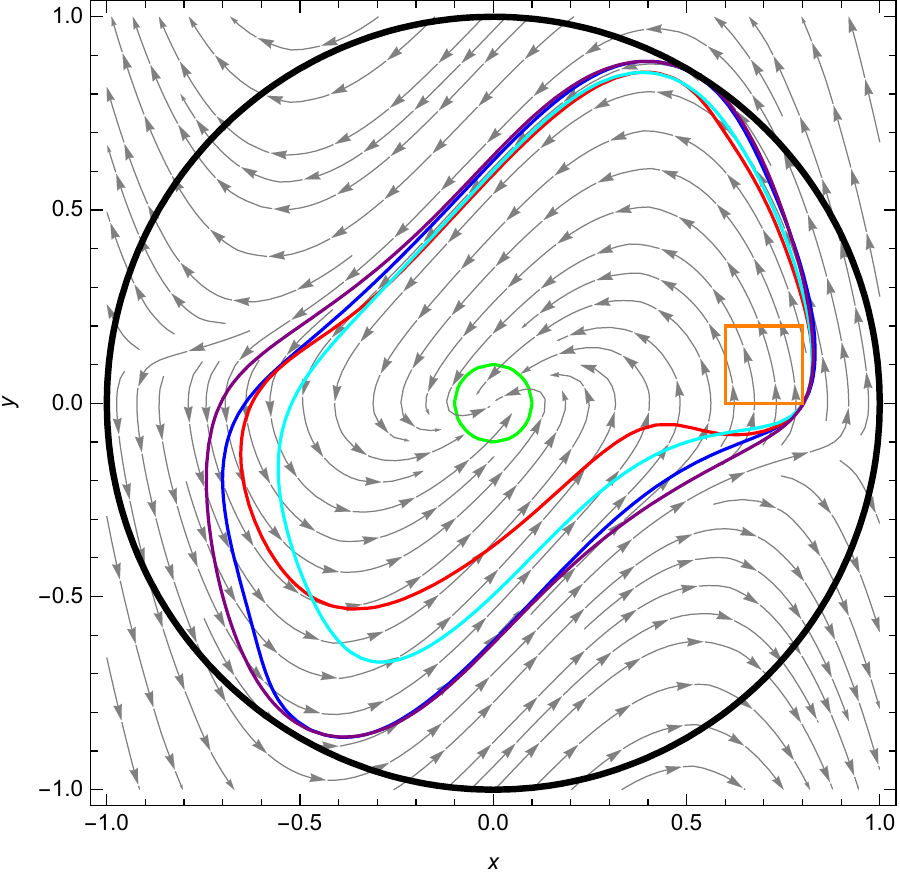} 
\caption{Green, black and orange curve- $\partial \mathcal{X}_r$, $\partial \mathcal{X}$ and $\partial \mathcal{X}_0$; blue, red, cyan and purple curve - $\partial \mathcal{R}$, which is computed via solving constraints \eqref{con4_00}-\eqref{con5_0} when the degree is 8, \eqref{con3110}-\eqref{con511} when the degree is 12 and $\beta=0.1$, \eqref{con4_000}-\eqref{con7_00} when the degree is 8 and $\beta=1$, and \eqref{con4_00_new}-\eqref{con5_0_new} when the degree is 6 and $\alpha(\bm{x})=x^2 v(\bm{x})$.}
\label{fig_2}
\end{figure}

\end{example}

\begin{example}
Consider the following system from \cite{tan2008stability},
\begin{equation*}
    \begin{cases}
           &\dot{x}=-0.42 x - 1.05 y - 2.3 x^2 - 0.56 x y - x^3\\
           &\dot{y}=1.98 x + x y
    \end{cases}
\end{equation*}
with $\mathcal{X}=\{\bm{x}\in \mathbb{R}^2\mid (x)^2 + (y)^2 - 4<0\}$, $\mathcal{X}_0=\{\bm{x}\in \mathbb{R}^2\mid (x - 1.2)^2 + (y - 0.8)^2 - 0.1<0\}$ and $\mathcal{X}_r=\{\bm{x}\in \mathbb{R}^2\mid (x + 1.2)^2 + (y + 0.5)^2 - 0.3<0\}$.

The reach-avoid property in the sense of Definition \ref{RNS} is not verified using constraint \eqref{041}-\eqref{043}. Actually, it cannot be verified via solving constraint \eqref{041}-\eqref{043}, since there exists an equilibrium in the set $\overline{\mathcal{X}\setminus \mathcal{X}_r}$.  

The reach-avoid property is verified when the degree is 10 for constraints \eqref{con4_00}-\eqref{con5_0}, and \eqref{con3110}-\eqref{con511} with $\beta=1$ and $\beta=0.1$. If constraint\eqref{con4_00_new}-\eqref{con5_0_new} is used with $\alpha(\bm{x})=(2-y)v(\bm{x})$, the degree is 8.  Some of the computed sets $\mathcal{R}$ are illustrated in Fig. \ref{fig_4}.

Like Example \ref{ex2}, we also experimented using constraints \eqref{con3110}-\eqref{con511} and \eqref{con4_000}-\eqref{con7_00} with $\beta=10^{-2},10^{-3},10^{-4},10^{-5},10^{-6}$ for more comprehensive and fair comparisons. The computations terminate when the degree takes $10$ for all of these experiments. 

\begin{figure}[htb!]
\center
\includegraphics[width=2.5in,height=2.0in]{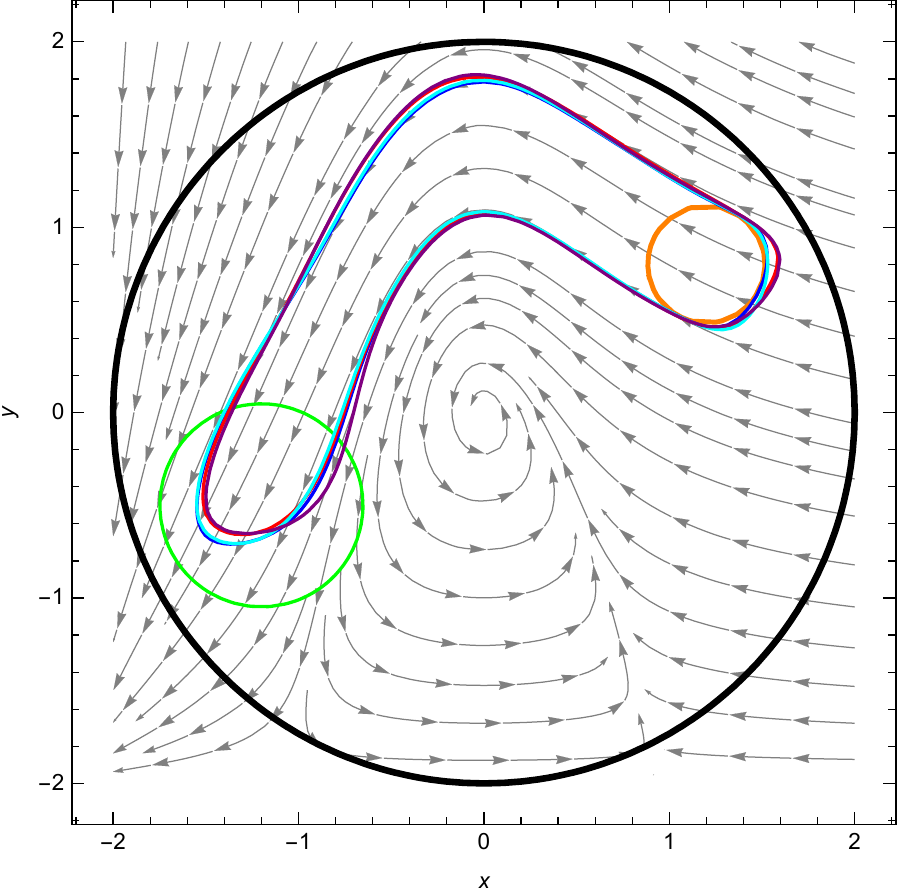} 
\caption{Green, black and orange curve- $\partial \mathcal{X}_r$, $\partial \mathcal{X}$ and $\partial \mathcal{X}_0$; blue, red, cyan and purple curve - $\partial \mathcal{R}$, which is computed via solving constraints \eqref{con4_00}-\eqref{con5_0} when the degree is $10$, \eqref{con3110}-\eqref{con511} when the degree is $10$ and $\beta=1$, \eqref{con3110}-\eqref{con511} when the degree is $10$ and $\beta=0.1$, and\eqref{con4_00_new}-\eqref{con5_0_new} when the degree is $8$ and $\alpha(\bm{x})=(2-y) v(\bm{x})$.}
\label{fig_4}
\end{figure}

\end{example}

\begin{example}
\label{ex7}
Consider the following system from \cite{tan2008stability},
\begin{equation*}
    \begin{cases}
           &\dot{x}=y\\
           &\dot{y}=-(1 - x^2)x - y
    \end{cases}
\end{equation*}
with $\mathcal{X}=\{\bm{x}\in \mathbb{R}^2\mid \frac{x^2}{4}+\frac{y^2}{9}-1<0\}$, $\mathcal{X}_0=\{\bm{x}\in \mathbb{R}^2\mid (x + 1)^2 + (y - 1.5)^2 - 0.25<0\}$ and $\mathcal{X}_r=\{\bm{x}\in \mathbb{R}^2\mid x^2+y^2-0.01<0\}$.

The reach-avoid property is not verified using constraint \eqref{041}-\eqref{043}, and is verified when the degree is 10 for constraint \eqref{con4_00}-\eqref{con5_0}. We did not obtain any positive verification result from solving constraint \eqref{con3110}-\eqref{con511} with $\beta=1$ and $\beta=0.1$. However, this negative situation can be improved by solving constraint \eqref{con4_000}-\eqref{con7_00} with $\beta=1$ and $\beta=0.1$, and the reach-avoid property is verified when the degree is 10. If constraint\eqref{con4_00_new}-\eqref{con5_0_new} is used with $\alpha(\bm{x})=(x+y)^2 v(\bm{x})$, the degree is 6. Furthermore, if constraint \eqref{con4_00_new}-\eqref{con5_0_new} is used with $\alpha(\bm{x})=x^4 v(\bm{x})$, the degree is 4.

Analogously, we also experimented using constraint \eqref{con3110}-\eqref{con511} and \eqref{con4_000}-\eqref{con7_00} with $\beta=10^{-2},10^{-3},10^{-4},10^{-5},10^{-6}$ for more comprehensive and fair comparisons. The computations terminate when the degree is $10$ for all of these experiments. 

\begin{figure}[htb!]
\center
\includegraphics[width=2.5in,height=2.0in]{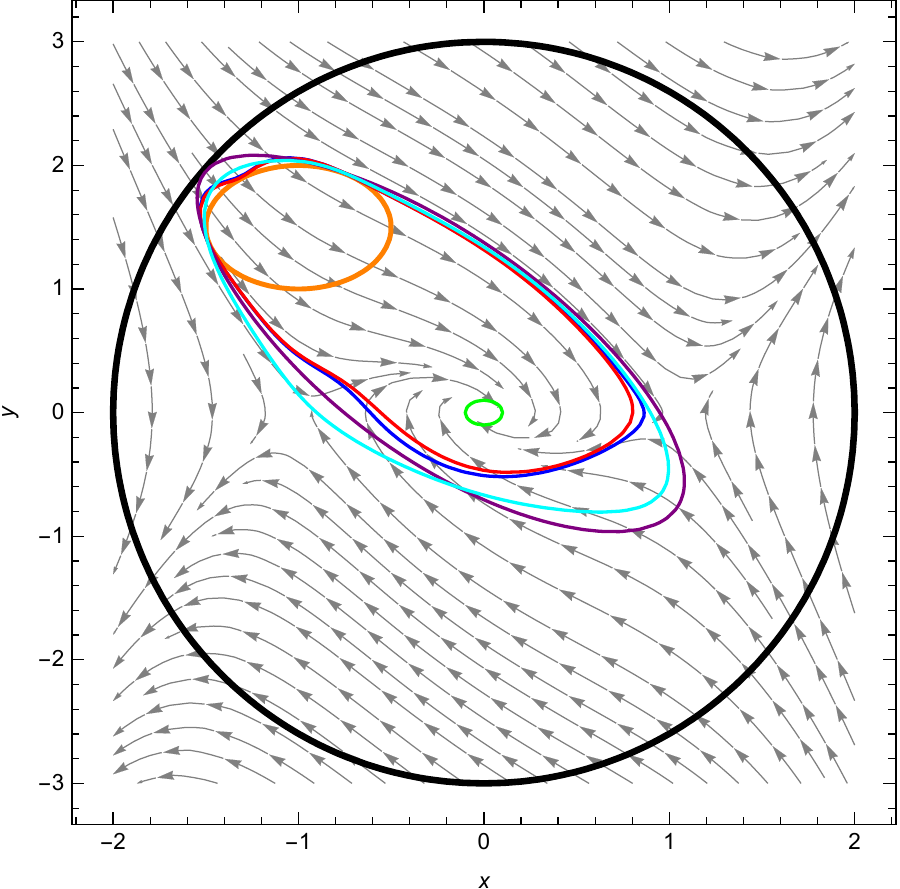} 
\caption{Green, black and orange curve- $\partial \mathcal{X}_r$, $\partial \mathcal{X}$ and $\partial \mathcal{X}_0$; blue, red, cyan and purple curve - $\partial \mathcal{R}$, which is computed via solving constraints \eqref{con4_00}-\eqref{con5_0} when the degree is $10$, \eqref{con4_000}-\eqref{con7_00} when the degree is $10$ and $\beta=1$, \eqref{con4_00_new}-\eqref{con5_0_new} when the degree is $6$ and $\alpha(\bm{x})=(x+y)^2 v(\bm{x})$, and\eqref{con4_00_new}-\eqref{con5_0_new} when the degree is $4$ and $\alpha(\bm{x})=x^4 v(\bm{x})$.}
\label{fig_3}
\end{figure}

\end{example}

\begin{example}[Dubin's Car]
\label{dubin}
Consider the Dubin's car, a multi-input system: $\dot{a}=v\cos(\theta)$, $\dot{b}=v\sin(\theta)$, $\dot{\theta}=\omega$. By the change of variables, $x=\theta$, $y=a\cos(\theta)+b\sin(\theta)$, $z=-2(a\sin(\theta)-b\cos(\theta))+\theta y$, and $u_1=\omega$, $u_2=v-\omega(a\sin(\theta)-b\cos(\theta))$, it is transformed into polynomial dynamics:
\begin{equation}
\begin{cases}
       &\dot{x}=u_1,\\
       &\dot{y}=u_2,\\
       &\dot{z}=y u_1 -x u_2.
    \end{cases}
\end{equation}
with $u_1=2,u_2=1+z-xy$, $\mathcal{X}=\{\bm{x}\in \mathbb{R}^3\mid x^2+y^2+z^2-4<0\}$, $\mathcal{X}_0=\{\bm{x}\in \mathbb{R}^3\mid (x + 0.6)^2 + y^2 + (z + 0.6)^2 - 0.02<0\}$ and $\mathcal{X}_r=\{\bm{x}\in \mathbb{R}^3\mid x^2+y^2+z^2-4<0,(x - 1.0)^2 - (y + 0.5)^2 + (z+0.1)^2 - 0.1<0\}$. 

The reach-avoid property is verified when degree is 8 for all of constraints \eqref{041}-\eqref{043}, \eqref{con4_00}-\eqref{con5_0}, \eqref{con3110}-\eqref{con511} with $\beta\in \{1,0.1,\ldots,10^{-6}\}$, \eqref{con4_000}-\eqref{con7_00} with $\beta\in \{1,0.1,\ldots,10^{-6}\}$. However, if constraint\eqref{con4_00_new}-\eqref{con5_0_new} is used with $\alpha(\bm{x})=(1-x)^2v(\bm{x})$, the degree is 6.   
\end{example}

 Examples above, i.e., Example \ref{ex1}-\ref{dubin}, indicate that when the discount factor is small, constraint \eqref{con3110}-\eqref{con511} has the same performance with constraint \eqref{con4_00}-\eqref{con5_0}, although it performs worse when the discount factor is large. On the other hand, constraint \eqref{con4_000}-\eqref{con7_00} indeed is able to improve the performance of constraint \eqref{con3110}-\eqref{con511} when the discount factor is large, but it does not improve constraint \eqref{con4_00}-\eqref{con5_0} and its performance will be the same with constraint \eqref{con3110}-\eqref{con511} when the discount factor is small. However, constraint \eqref{con4_00_new}-\eqref{con5_0_new} outperforms the former three, i.e., constraints \eqref{con3110}-\eqref{con511}, \eqref{con4_00}-\eqref{con5_0} and \eqref{con4_000}-\eqref{con7_00}, and constraint \eqref{041}-\eqref{043}. It is indeed more expressive and has more feasible solutions, providing more chances for verifying the reach-avoid property in the sense of Definition \ref{RNS} successfully. Besides, from Example \ref{ex7} we observe that the performance of constraint\eqref{con4_00_new}-\eqref{con5_0_new} is affected by the choice of the function $\alpha(\bm{x})$, and an appropriate choice will be more conducive to the reach-avoid verification. However, how to determine such a function in an optimal sense is still an open problem, which will be investigated in the future work. In practice, engineering experiences and insights may facilitate the choice. 
 
In the present work we only demonstrate the performance of all of quantified constraints by relaxing them into semi-definite constraints and addressing them within the semi-definite programming framework, which could be solved efficiently via interior point methods in polynomial time. It is worth remarking here that besides semi-definite programs for implementing these constraints, other methods such as counterexample-guided inductive synthesis methods combining machine learning and SMT solving techniques (e.g., \cite{abate2021fossil}) can also be used to solve these constraints. We did not show their performance in this present work and leave these investigations for ones of interest.

\section{Conclusion and Future Work}
\label{con}
In this paper we studied the reach-avoid verification problem of continuous-time systems within the framework of optimization based methods. At the beginning of our method, two sets of quantified inequalities were derived respectively based on a discount value function with the discount factor being larger than zero and equal to zero, such that the reach-avoid verification problem is transformed into a problem of searching for exponential/asymptotic guidance-barrier functions. The set of constraints associated with asymptotic guidance-barrier functions is completely novel and has certain benefits over the other one associated with exponential guidance-barrier functions, which is a simplified version of the one in existing literature. Furthermore, we enhanced the new set of constraints such that it is more expressive than the aforementioned two sets of constraints, providing more chances to verify the satisfaction of reach-avoid properties successfully. When the datum involved are polynomials, i.e., the initial set, safe set and target set are semi-algebraic, and the system has polynomial dynamics, the problem of solving these sets of constraints can be efficiently addressed using convex optimization. Finally, several examples demonstrated the theoretical developments and benefits of the proposed approaches.

Like the extension of barrier functions for safety verification to control barrier functions for synthesizing safe controllers \cite{ames2016control}, in the future work we will extend exponential/asymptotic guidance-barrier functions to control exponential/asymptotic guidance-barrier functions for synthesizing provably reach-avoid controllers.





\bibliographystyle{abbrv}
\bibliography{reference}

\begin{thebibliography}{10}

\bibitem{abate2021fossil}
A.~Abate, D.~Ahmed, A.~Edwards, M.~Giacobbe, and A.~Peruffo.
\newblock Fossil: a software tool for the formal synthesis of lyapunov
  functions and barrier certificates using neural networks.
\newblock In {\em Proceedings of the 24th International Conference on Hybrid
  Systems: Computation and Control}, pages 1--11, 2021.

\bibitem{althoff2013reachability}
M.~Althoff.
\newblock Reachability analysis of nonlinear systems using conservative
  polynomialization and non-convex sets.
\newblock In {\em Proceedings of the 16th international conference on Hybrid
  systems: computation and control}, pages 173--182, 2013.

\bibitem{althoff2021set}
M.~Althoff, G.~Frehse, and A.~Girard.
\newblock Set propagation techniques for reachability analysis.
\newblock {\em Annual Review of Control, Robotics, and Autonomous Systems},
  4(1), 2021.

\bibitem{ames2019control}
A.~D. Ames, S.~Coogan, M.~Egerstedt, G.~Notomista, K.~Sreenath, and P.~Tabuada.
\newblock Control barrier functions: Theory and applications.
\newblock In {\em 2019 18th European control conference (ECC)}, pages
  3420--3431. IEEE, 2019.

\bibitem{ames2016control}
A.~D. Ames, X.~Xu, J.~W. Grizzle, and P.~Tabuada.
\newblock Control barrier function based quadratic programs for safety critical
  systems.
\newblock {\em IEEE Transactions on Automatic Control}, 62(8):3861--3876, 2016.

\bibitem{asarin2000}
E.~Asarin, O.~Bournez, T.~Dang, and O.~Maler.
\newblock Approximate reachability analysis of piecewise-linear dynamical
  systems.
\newblock In {\em HSCC'00}, pages 20--31. Springer, 2000.

\bibitem{aubin2001viability}
J.-P. Aubin.
\newblock Viability kernels and capture basins of sets under differential
  inclusions.
\newblock {\em SIAM Journal on Control and Optimization}, 40(3):853--881, 2001.

\bibitem{bak2018t}
S.~Bak.
\newblock t-barrier certificates: a continuous analogy to k-induction.
\newblock {\em IFAC-PapersOnLine}, 51(16):145--150, 2018.

\bibitem{berz1998verified}
M.~Berz and K.~Makino.
\newblock Verified integration of odes and flows using differential algebraic
  methods on high-order taylor models.
\newblock {\em Reliable computing}, 4(4):361--369, 1998.

\bibitem{bokanowski2010}
O.~Bokanowski, N.~Forcadel, and H.~Zidani.
\newblock Reachability and minimal times for state constrained nonlinear
  problems without any controllability assumption.
\newblock {\em SIAM Journal on Control and Optimization}, 48(7):4292--4316,
  2010.

\bibitem{braunl2008}
T.~Br{\"a}unl.
\newblock {\em Embedded Robotics: Mobile Robot Design and Applications with
  Embedded systems}.
\newblock Springer Science \& Business Media, 2008.

\bibitem{chen2012taylor}
X.~Chen, E.~Abraham, and S.~Sankaranarayanan.
\newblock Taylor model flowpipe construction for non-linear hybrid systems.
\newblock In {\em 2012 IEEE 33rd Real-Time Systems Symposium}, pages 183--192.
  IEEE, 2012.

\bibitem{dai2017barrier}
L.~Dai, T.~Gan, B.~Xia, and N.~Zhan.
\newblock Barrier certificates revisited.
\newblock {\em Journal of Symbolic Computation}, 80:62--86, 2017.

\bibitem{eggers2015improving}
A.~Eggers, N.~Ramdani, N.~S. Nedialkov, and M.~Fr{\"a}nzle.
\newblock Improving the sat modulo ode approach to hybrid systems analysis by
  combining different enclosure methods.
\newblock {\em Software \& Systems Modeling}, 14(1):121--148, 2015.

\bibitem{fisac2015}
J.~F. Fisac, M.~Chen, C.~J. Tomlin, and S.~S. Sastry.
\newblock Reach-avoid problems with time-varying dynamics, targets and
  constraints.
\newblock In {\em HSCC'15}, pages 11--20, 2015.

\bibitem{fisac2015pursuit}
J.~F. Fisac and S.~S. Sastry.
\newblock The pursuit-evasion-defense differential game in dynamic constrained
  environments.
\newblock In {\em CDC'15}, pages 4549--4556. IEEE, 2015.

\bibitem{franzle2019memory}
M.~Fr{\"a}nzle, M.~Chen, and P.~Kr{\"o}ger.
\newblock In memory of oded maler: automatic reachability analysis of
  hybrid-state automata.
\newblock {\em ACM SIGLOG News}, 6(1):19--39, 2019.

\bibitem{Zhan18}
T.~Gan, M.~Chen, Y.~Li, B.~Xia, and N.~Zhan.
\newblock Reachability analysis for solvable dynamical systems.
\newblock {\em IEEE Trans. Automat. Contr.}, 63(7):2003--2018, 2018.

\bibitem{gillula2012}
J.~H. Gillula and C.~J. Tomlin.
\newblock Guaranteed safe online learning via reachability: tracking a ground
  target using a quadrotor.
\newblock In {\em ICRA'12}, pages 2723--2730. IEEE, 2012.

\bibitem{henrion2013convex}
D.~Henrion and M.~Korda.
\newblock Convex computation of the region of attraction of polynomial control
  systems.
\newblock {\em IEEE Trans. Automat. Contr.}, 59(2):297--312, 2013.

\bibitem{Henzinger98}
T.~A. Henzinger, P.~W. Kopke, A.~Puri, and P.~Varaiya.
\newblock What's decidable about hybrid automata?
\newblock {\em J. Comput. Syst. Sci.}, 57(1):94--124, 1998.

\bibitem{khalil2002nonlinear}
H.~K. Khalil.
\newblock Nonlinear systems third edition.
\newblock {\em Patience Hall}, 115, 2002.

\bibitem{kong2013exponential}
H.~Kong, F.~He, X.~Song, W.~N. Hung, and M.~Gu.
\newblock Exponential-condition-based barrier certificate generation for safety
  verification of hybrid systems.
\newblock In {\em International Conference on Computer Aided Verification},
  pages 242--257. Springer, 2013.

\bibitem{korda2013inner}
M.~Korda, D.~Henrion, and C.~N. Jones.
\newblock Inner approximations of the region of attraction for polynomial
  dynamical systems.
\newblock {\em IFAC Proceedings Volumes}, 46(23):534--539, 2013.

\bibitem{korda2014controller}
M.~Korda, D.~Henrion, and C.~N. Jones.
\newblock Controller design and region of attraction estimation for nonlinear
  dynamical systems.
\newblock {\em IFAC Proceedings Volumes}, 47(3):2310--2316, 2014.

\bibitem{kulisch2001perspectives}
U.~Kulisch, R.~Lohner, and A.~Facius.
\newblock {\em Perspectives on Enclosure Methods}.
\newblock Springer Science \& Business Media, 2001.

\bibitem{lee2008cyber}
E.~A. Lee.
\newblock Cyber physical systems: Design challenges.
\newblock In {\em ISORC'18}, pages 363--369. IEEE, 2008.

\bibitem{lofberg2004}
J.~Lofberg.
\newblock Yalmip: A toolbox for modeling and optimization in matlab.
\newblock In {\em CACSD'04}, pages 284--289. IEEE, 2004.

\bibitem{majumdar2014}
A.~Majumdar, R.~Vasudevan, M.~M. Tobenkin, and R.~Tedrake.
\newblock Convex optimization of nonlinear feedback controllers via occupation
  measures.
\newblock {\em The International Journal of Robotics Research},
  33(9):1209--1230, 2014.

\bibitem{margellos2011}
K.~Margellos and J.~Lygeros.
\newblock Hamilton--jacobi formulation for reach--avoid differential games.
\newblock {\em IEEE Trans. Automat. Contr.}, 56(8):1849--1861, 2011.

\bibitem{margellos2012}
K.~Margellos and J.~Lygeros.
\newblock Toward 4-d trajectory management in air traffic control: a study
  based on monte carlo simulation and reachability analysis.
\newblock {\em IEEE Transactions on Control Systems Technology},
  21(5):1820--1833, 2012.

\bibitem{mitchell2000level}
I.~Mitchell and C.~J. Tomlin.
\newblock Level set methods for computation in hybrid systems.
\newblock In {\em HSCC'00}, pages 310--323. Springer, 2000.

\bibitem{mosek2015mosek}
A.~Mosek.
\newblock The mosek optimization toolbox for matlab manual, 2015.

\bibitem{prajna2004safety}
S.~Prajna and A.~Jadbabaie.
\newblock Safety verification of hybrid systems using barrier certificates.
\newblock In {\em HSCC'04}, pages 477--492. Springer, 2004.

\bibitem{prajna2007framework}
S.~Prajna, A.~Jadbabaie, and G.~J. Pappas.
\newblock A framework for worst-case and stochastic safety verification using
  barrier certificates.
\newblock {\em IEEE Transactions on Automatic Control}, 52(8):1415--1428, 2007.

\bibitem{prajna2007convex}
S.~Prajna and A.~Rantzer.
\newblock Convex programs for temporal verification of nonlinear dynamical
  systems.
\newblock {\em SIAM Journal on Control and Optimization}, 46(3):999--1021,
  2007.

\bibitem{ramdani2009hybrid}
N.~Ramdani, N.~Meslem, and Y.~Candau.
\newblock A hybrid bounding method for computing an over-approximation for the
  reachable set of uncertain nonlinear systems.
\newblock {\em IEEE Transactions on automatic control}, 54(10):2352--2364,
  2009.

\bibitem{shia2014convex}
V.~Shia, R.~Vasudevan, R.~Bajcsy, and R.~Tedrake.
\newblock Convex computation of the reachable set for controlled polynomial
  hybrid systems.
\newblock In {\em CDC'14}, pages 1499--1506. IEEE, 2014.

\bibitem{sogokon2018vector}
A.~Sogokon, K.~Ghorbal, Y.~K. Tan, and A.~Platzer.
\newblock Vector barrier certificates and comparison systems.
\newblock In {\em International Symposium on Formal Methods}, pages 418--437.
  Springer, 2018.

\bibitem{tan2008stability}
W.~Tan and A.~Packard.
\newblock Stability region analysis using polynomial and composite polynomial
  lyapunov functions and sum-of-squares programming.
\newblock {\em IEEE Trans. Automat. Contr.}, 53(2):565--571, 2008.

\bibitem{williams2003}
B.~C. Williams, M.~D. Ingham, S.~H. Chung, and P.~H. Elliott.
\newblock Model-based programming of intelligent embedded systems and robotic
  space explorers.
\newblock {\em Proceedings of the IEEE}, 91(1):212--237, 2003.

\bibitem{xue2016reach}
B.~Xue, A.~Easwaran, N.-J. Cho, and M.~Fr{\"a}nzle.
\newblock Reach-avoid verification for nonlinear systems based on boundary
  analysis.
\newblock {\em IEEE Trans. Automat. Contr.}, 62(7):3518--3523, 2016.

\bibitem{xue2019}
B.~Xue, M.~Fr\"anzle, and N.~Zhan.
\newblock Inner-approximating reachable sets for polynomial systems with
  time-varying uncertainties.
\newblock {\em IEEE Trans. Automat. Contr.}, 65(4):1468--1483, 2020.

\bibitem{xue2020inner}
B.~Xue, N.~Zhan, and M.~Fr{\"a}nzle.
\newblock Inner-approximating reach-avoid sets for discrete-time polynomial
  systems.
\newblock In {\em CDC'20}, pages 867--873. IEEE, 2020.

\bibitem{zhao2017control}
P.~Zhao, S.~Mohan, and R.~Vasudevan.
\newblock Control synthesis for nonlinear optimal control via convex
  relaxations.
\newblock In {\em ACC'17}, pages 2654--2661. IEEE, 2017.

\end{thebibliography}
\section{Appendix}

\begin{algorithm}
The semi-definite program for solving constraints \eqref{041}-\eqref{043}:
\begin{equation*}
\begin{split}
&-v(\bm{x})+\sum_{i=1}^{k_1}s_{0,i}(\bm{x})l_i(\bm{x}) \in \sum[\bm{x}],\\
&-\bigtriangledown_{\bm{x}} v(\bm{x})\cdot\bm{f}(\bm{x})-\epsilon+s_{1,j}(\bm{x}) h(\bm{x})-s_{2,j}(\bm{x}) g_j(\bm{x})\in \sum[\bm{x}],\\
&v(\bm{x})-\epsilon-q(\bm{x}) h(\bm{x})\in \sum[\bm{x}],\\
& j=1,\ldots,j_1,
\end{split}
\end{equation*}
where $v(\bm{x})$, $q(\bm{x})\in \mathbb{R}[\bm{x}]$, and $s_{0,i}(\bm{x})\in \sum[\bm{x}]$, $i=0,\ldots,k_1$, $s_{1,j}, s_{2,j}\in \sum[\bm{x}]$, $j=1,\ldots,j_1$, and $\epsilon>0$ is an user-defined value enforcing the strict positiveness of $v(\bm{x})$ and $-\bigtriangledown_{\bm{x}} v(\bm{x})\cdot\bm{f}(\bm{x})$. In the computations of this paper we take $\epsilon=10^{-6}$.
\end{algorithm}

\begin{algorithm}

The semi-definite program for solving constraints \eqref{con3110}-\eqref{con511}: 
\begin{equation*}
\begin{split}
&v(\bm{x})-\epsilon+\sum_{i=1}^{k_1}s_{0,i}(\bm{x})l_i(\bm{x}) \in \sum[\bm{x}],\\
&\bigtriangledown_{\bm{x}} v(\bm{x})\cdot\bm{f}(\bm{x})-\beta v(\bm{x})+s_{1,j}(\bm{x}) h(\bm{x})-s_{2,j}(\bm{x}) g_j(\bm{x})\in \sum[\bm{x}],\\
&-v(\bm{x})-p(\bm{x}) h(\bm{x})\in \sum[\bm{x}],\\
&j=1,\ldots,j_1,
\end{split}
\end{equation*}
where $v(\bm{x}), q(\bm{x})\in \mathbb{R}[\bm{x}]$, and $s_{0,i}(\bm{x})\in \sum[\bm{x}]$, $i=0,\ldots,k_1$, $s_{1,j}, s_{2,j}\in \sum[\bm{x}]$, $j=1,\ldots,j_1$, and $\epsilon>0$ is an user-defined value enforcing the strict positiveness of $v(\bm{x})$. In the computations of this paper we take $\epsilon=10^{-6}$.
\end{algorithm}

\begin{algorithm}
The semi-definite program for solving constraints \eqref{con4_00}-\eqref{con5_0}:
\begin{equation*}
\begin{split}
&v(\bm{x})-\epsilon+\sum_{i=1}^{k_1}s_{0,i}(\bm{x})l_i(\bm{x}) \in \sum[\bm{x}], \\
&\bigtriangledown_{\bm{x}} v(\bm{x})\cdot\bm{f}(\bm{x})+s_{1,j}(\bm{x}) h(\bm{x})-s_{2,j}(\bm{x}) g_j(\bm{x})\in \sum[\bm{x}],\\
&\bigtriangledown_{\bm{x}} w(\bm{x})\cdot\bm{f}(\bm{x})-v(\bm{x})+s_{3,j}(\bm{x}) h(\bm{x})-s_{4,j}(\bm{x}) g_j(\bm{x})\in \sum[\bm{x}],\\
&-v(\bm{x})-p(\bm{x}) h(\bm{x}) \in \sum[\bm{x}],\\
&j=1,\ldots,j_1,\\
\end{split}
\end{equation*}
where $v(\bm{x}), w(\bm{x}), p(\bm{x})\in \mathbb{R}[\bm{x}]$, and $s_{0,i}(\bm{x})\in \sum[\bm{x}]$, $i=1,\ldots,k_1$, and $s_{i,j} \in \sum[\bm{x}]$, $i=1,\ldots,4, j=1,\ldots,j_1$, $\epsilon>0$ is an user-defined value enforcing the strict positiveness of $v(\bm{x})$ over the initial set $\mathcal{X}_0$. In the computations of this paper we take $\epsilon=10^{-6}$.
\end{algorithm}

\begin{algorithm}
The semi-definite program for solving constraints \eqref{con4_000}-\eqref{con7_00}:
\begin{equation*}
\begin{split}
&v_1(\bm{x})+v_2(\bm{x})-\epsilon+\sum_{i=1}^{k_1}s_{0,i}(\bm{x})l_i(\bm{x}) \in \sum[\bm{x}], \\
&\bigtriangledown_{\bm{x}} v_1(\bm{x})\cdot\bm{f}(\bm{x})+s_{1,j}(\bm{x}) h(\bm{x})-s_{2,j}(\bm{x}) g_j(\bm{x})\in \sum[\bm{x}],\\
&\bigtriangledown_{\bm{x}} w(\bm{x})\cdot\bm{f}(\bm{x})-v_1(\bm{x})+s_{3,j}(\bm{x}) h(\bm{x})-s_{4,j}(\bm{x}) g_j(\bm{x})\in \sum[\bm{x}],\\
&\bigtriangledown_{\bm{x}} v_2(\bm{x})\cdot\bm{f}(\bm{x})-\beta v_2(\bm{x})+s_{5,j}(\bm{x}) h(\bm{x})-s_{6,j}(\bm{x}) g_j(\bm{x})\in \sum[\bm{x}],\\
&-v_1(\bm{x})-p(\bm{x}) h(\bm{x}) \in \sum[\bm{x}],\\
&-v_2(\bm{x})-p(\bm{x}) h(\bm{x})\in \sum[\bm{x}],\\
&j=1,\ldots,j_1,\\
\end{split}
\end{equation*}
where $v(\bm{x}), w(\bm{x}), p(\bm{x})\in \mathbb{R}[\bm{x}]$, and $s_{0,i}(\bm{x})\in \sum[\bm{x}]$, $i=1,\ldots,k_1$, and $s_{i,j} \in \sum[\bm{x}]$, $i=1,\ldots,6, j=1,\ldots,j_1$, $\epsilon>0$ is an user-defined value enforcing the strict positiveness of $v(\bm{x})$ over the initial set $\mathcal{X}_0$. In the computations of this paper we take $\epsilon=10^{-6}$.
\end{algorithm}

\begin{algorithm}
The semi-definite program for solving constraints \eqref{con4_00_new}-\eqref{con5_0_new}:
\begin{equation*}
\begin{split}
&v(\bm{x})-\epsilon+\sum_{i=1}^{k_1}s_{0,i}(\bm{x})l_i(\bm{x}) \in \sum[\bm{x}], \\
&\bigtriangledown_{\bm{x}} v(\bm{x})\cdot\bm{f}(\bm{x})-\alpha(\bm{x})+s_{1,j}(\bm{x}) h(\bm{x})-s_{2,j}(\bm{x}) g_j(\bm{x})\in \sum[\bm{x}],\\
&\bigtriangledown_{\bm{x}} w(\bm{x})\cdot\bm{f}(\bm{x})-v(\bm{x})+s_{3,j}(\bm{x}) h(\bm{x})-s_{4,j}(\bm{x}) g_j(\bm{x})\in \sum[\bm{x}],\\
&-v(\bm{x})-p(\bm{x}) h(\bm{x}) \in \sum[\bm{x}],\\
&j=1,\ldots,j_1,\\
\end{split}
\end{equation*}
where $v(\bm{x}), w(\bm{x}), p(\bm{x})\in \mathbb{R}[\bm{x}]$, and $s_{0,i}(\bm{x})\in \sum[\bm{x}]$, $i=1,\ldots,k_1$, and $s_{i,j} \in \sum[\bm{x}]$, $i=1,\ldots,4, j=1,\ldots,j_1$, $\epsilon>0$ is an user-defined value enforcing the strict positiveness of $v(\bm{x})$ over the initial set $\mathcal{X}_0$. In the computations of this paper we take $\epsilon=10^{-6}$.
\end{algorithm}

\clearpage

\begin{IEEEbiography}[{\includegraphics[width=1in,height=1.25in,clip,keepaspectratio]{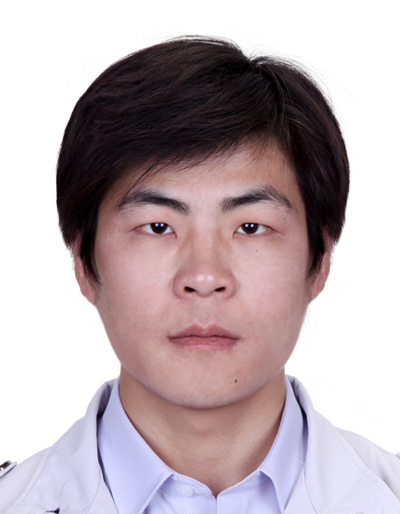}}]{Bai Xue}
received the B.Sc. degree in information and computing science from Tianjin University of Technology and Education, Tianjin, China, in 2008, and the Ph.D. degree in applied mathematics from Beihang University, Beijing, China,
in 2014.

He is currently a Research Professor with the Institute of Software, Chinese Academy of Sciences, Beijing, China. Prior to joining the Institute of Software, he worked as a Research Fellow with the Centre for High Performance Embedded Systems, Nanyang Technological University, from May, 2014 to September, 2015, and as a postdoctoral with the Department f\"ur Informatik, Carl von Ossietzky Universit\"at Oldenburg, from November, 2015 to October, 2017. His research interests involve formal verification of (stochastic/time-delay)
hybrid systems and AI.
\end{IEEEbiography}

\begin{IEEEbiography}[{\includegraphics[width=1in,height=1.25in,clip,keepaspectratio]{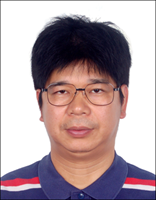}}]{Naijun Zhan}
 received the B.Sc. degree in mathematics and the M.Sc. degree in computer science from Nanjing University, Nanjing, China, in
1993 and 1996, respectively, and the Ph.D. degree in computer science from the Institute of Software, Chinese Academy of Sciences, Beijing, China, in 2000.

He was with the Faculty of Mathematics and Information, University of Mannheim,
Mannheim, Germany, as a research fellow, from
2001 to 2004. Since then, he joined the Institute
of Software, Chinese Academy of Sciences, Beijing, China, as an Associate Research Professor, and later was promoted to be a Full Research
Professor in 2008, and a Distinguished Professor in 2015. His research
interests include real-time, embedded and hybrid systems, program verification, concurrent computation models, and modal and temporal logics.
\end{IEEEbiography}
\begin{IEEEbiography}[{\includegraphics[width=1in,height=1.25in,clip,keepaspectratio]{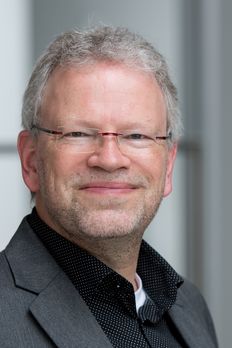}}]{Martin Fr\"anzle}
received the B.Sc. and M.Sc.
degrees in informatics from the Department
of Informatics at Christian-Albrechts-Universitat¨
Kiel, Germany, in 1986 and 1991, respectively,
and the Ph.D. degree in nature sciences from
the Technical Faculty of the Christian-AlbrechtsUniversitat Kiel, Germany, in 1997. 
He is currently a Full Professor for Foundations and Applications of Cyber-Physical Systems with the Department of Computing Science, Carl von Ossietzky Universität Oldenburg, Germany. At his university as well as at international Ph.D. schools,
he teaches in the areas of formal methods for embedded system design,
the theory and application of hybrid discrete-continuous systems, design
principles of autonomous systems, as well as related subjects.

His research interests include modeling, verification, and synthesis of
reactive, real-time, hybrid, and human-cyber-physical systems, as well
as applications in advanced driver assistance, highly automated cars
and autonomous driving, and power networks.
\end{IEEEbiography}

\begin{IEEEbiography}[{\includegraphics[width=1in,height=1.25in,clip,keepaspectratio]{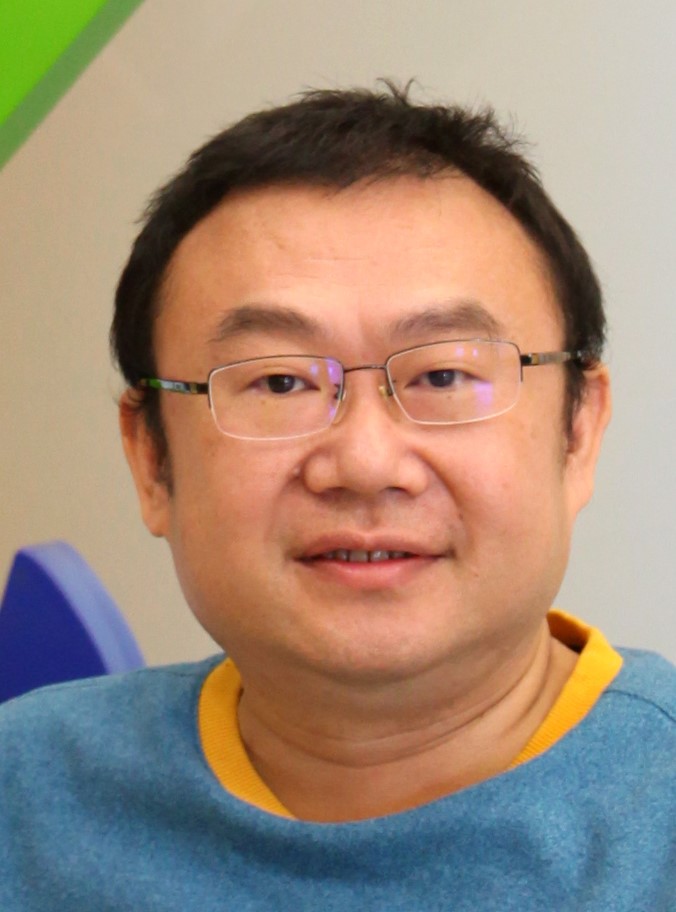}}]{Ji Wang}
 received his Ph.D. degree in computer science from National University of Defense Technology, China in 1995. He is currently a professor in State Key Laboratory of High Performance Computing and College of Computer Science and Technology at National University of Defense Technology, China. 
 
 His research interests include programming methodology, program analysis and verification, formal methods and software engineering.
\end{IEEEbiography}

\begin{IEEEbiography}[{\includegraphics[width=1in,height=1.25in,clip,keepaspectratio]{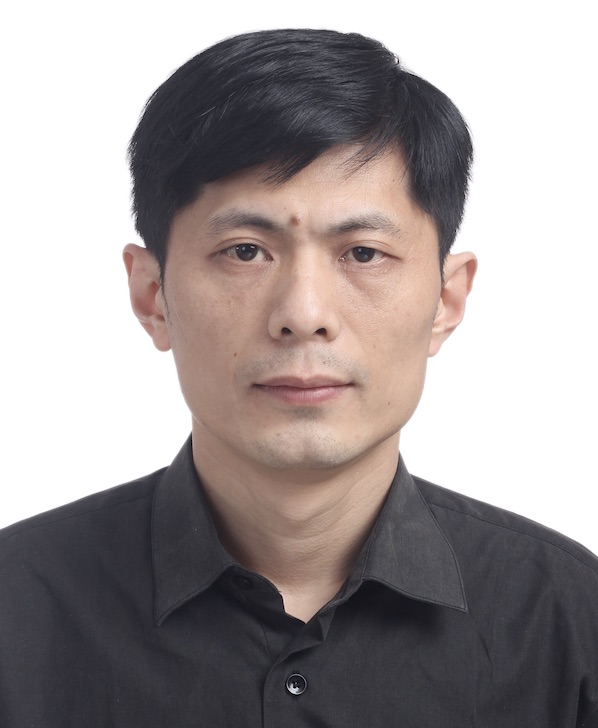}}]{Wanwei Liu} received his Ph.D degree in National University of Defense Technology in 2009, he is a full professor in College of Computer Science, National University of Defense Technology. He is a senior member of CCF.

His research interests includes theoretical computer science (particularly in automata theory and temporal logic), formal methods (particularly in verification), and software engineering.
\end{IEEEbiography}

\end{document}